\documentclass[12pt]{amsart}
\usepackage{amsmath,amsthm,amsfonts,amssymb,mathrsfs}
\date{\today}
\usepackage{color}
\usepackage{hyperref}
\newtheorem{theorem}{Theorem}[section]
\newtheorem{proposition}[theorem]{Proposition}
\newtheorem{corollary}[theorem]{Corollary}
\newtheorem{lemma}[theorem]{Lemma}
\theoremstyle{definition}

\newtheorem{example}[theorem]{Example}
\newtheorem{remark}[theorem]{Remark}

\begin{document}

\title[On the closure of the extended bicyclic semigroup]{On
the closure of the extended bicyclic semigroup}

\author[I. Fihel]{Iryna~Fihel}

\address{Department of Mechanics and Mathematics, Ivan Franko
National University of Lviv, Universytetska 1, Lviv, 79000, Ukraine}
\email{figel.iryna@gmail.com}

\author[O. Gutik]{Oleg~Gutik}

\address{Department of Mechanics and Mathematics, Ivan Franko
National University of Lviv, Universytetska 1, Lviv, 79000, Ukraine}
\email{o\_\,gutik@franko.lviv.ua, ovgutik@yahoo.com}



\begin{abstract}
In the paper we study the semigroup $\mathscr{C}_{\mathbb{Z}}$ which
is a generalization of the bicyclic semigroup. We describe main
algebraic properties of the semigroup $\mathscr{C}_{\mathbb{Z}}$ and
prove that every non-trivial congruence $\mathfrak{C}$ on the
semigroup $\mathscr{C}_{\mathbb{Z}}$ is a group congruence, and
moreover the quotient semigroup
$\mathscr{C}_{\mathbb{Z}}/\mathfrak{C}$ is isomorphic to a cyclic
group. Also we show that the semigroup $\mathscr{C}_{\mathbb{Z}}$ as
a Hausdorff semitopological semigroup admits only the discrete
topology. Next we study the closure
$\operatorname{cl}_T\left(\mathscr{C}_{\mathbb{Z}}\right)$ of the
semigroup $\mathscr{C}_{\mathbb{Z}}$ in a topological semigroup $T$.
We show that the non-empty remainder of $\mathscr{C}_{\mathbb{Z}}$
in a topological inverse semigroup $T$ consists of a group of units
$H(1_T)$ of $T$ and a two-sided ideal $I$ of $T$ in the case when
$H(1_T)\neq\varnothing$ and $I\neq\varnothing$. In the case when $T$
is a locally compact topological inverse semigroup and
$I\neq\varnothing$ we prove that an ideal $I$ is topologically
isomorphic to the discrete additive group of integers and describe
the topology on the subsemigroup $\mathscr{C}_{\mathbb{Z}}\cup I$.
Also we show that if the group of units $H(1_T)$ of the semigroup
$T$ is non-empty, then $H(1_T)$ is either singleton or $H(1_T)$ is
topologically isomorphic to the discrete additive group of integers.
\end{abstract}

\maketitle


\section{Introduction and preliminaries}

In this paper all topological spaces are assumed to be Hausdorff. We
shall follow the terminology of~\cite{CHK, CP, Engelking1989,
GHKLMS}. If $Y$ is a subspace of a topological space $X$ and
$A\subseteq Y$, then by $\operatorname{cl}_Y(A)$ we shall denote the
topological closure of $A$ in $Y$.  We denote by $\mathbb{N}$ the
set of positive integers.

An algebraic semigroup $S$ is called {\it inverse} if for any
element $x\in S$ there exists the unique $x^{-1}\in S$ such that
$xx^{-1}x=x$ and $x^{-1}xx^{-1}=x^{-1}$. The element $x^{-1}$ is
called the {\it inverse of} $x\in S$. If $S$ is an inverse
semigroup, then the function $\operatorname{inv}\colon S\to S$ which
assigns to every element $x$ of $S$ its inverse element $x^{-1}$ is
called an {\it inversion}.

A congruence $\mathfrak{C}$ on a semigroup $S$ is called
\emph{non-trivial} if $\mathfrak{C}$ is distinct from universal and
identity congruence on $S$, and \emph{group} if the quotient
semigroup $S/\mathfrak{C}$ is a group.

If $S$ is a semigroup, then we shall denote the subset of
idempotents in $S$ by $E(S)$. If $S$ is an inverse semigroup, then
$E(S)$ is closed under multiplication and we shall refer to $E(S)$ a
\emph{band} (or the \emph{band of} $S$). If the band $E(S)$ is a
non-empty subset of $S$, then the semigroup operation on $S$
determines the following partial order $\leqslant$ on $E(S)$:
$e\leqslant f$ if and only if $ef=fe=e$. This order is called the
{\em natural partial order} on $E(S)$. A \emph{semilattice} is a
commutative semigroup of idempotents. A semilattice $E$ is called
{\em linearly ordered} or a \emph{chain} if its natural order is a
linear order.

Let $E$ be a semilattice and $e\in E$. We denote ${\downarrow} e=\{
f\in E\mid f\leqslant e\}$ and ${\uparrow} e=\{ f\in E\mid
e\leqslant f\}$.

If $S$ is a semigroup, then we shall denote by $\mathscr{R}$,
$\mathscr{L}$, $\mathscr{D}$ and $\mathscr{H}$ the Green relations
on $S$ (see \cite{CP}):
\begin{align*}
    &\qquad a\mathscr{R}b \mbox{ if and only if } aS^1=bS^1;\\
    &\qquad a\mathscr{L}b \mbox{ if and only if } S^1a=S^1b;\\
    &\qquad a\mathscr{J}b \mbox{ if and only if } S^1aS^1=S^1bS^1;\\
    &\qquad \mathscr{D}=\mathscr{L}\circ\mathscr{R}=\mathscr{R}
    \circ\mathscr{L};\\
    &\qquad \mathscr{H}=\mathscr{L}\cap\mathscr{R}.
\end{align*}

A semigroup $S$ is called \emph{simple} if $S$ does not contain
proper two-sided ideals and \emph{bisimple} if $S$ has only one
$\mathscr{D}$-class.

A {\it semitopological} (resp. \emph{topological}) {\it semigroup}
is a Hausdorff topological space together with a separately (resp.
jointly) continuous semigroup operation \cite{CHK, Ruppert1984}. An
inverse topological semigroup with the continuous inversion is
called a \emph{topological inverse semigroup}. A topology $\tau$ on
a (inverse) semigroup $S$ which turns $S$ to be a topological
(inverse) semigroup is called a (\emph{inverse}) \emph{semigroup
topology} on $S$.

An element $s$ of a topological semigroup $S$ is called
\emph{topologically periodic} if for every open neighbourhood $U(s)$
of $s$ in $S$ there exists a positive integer $n\geqslant 2$ such
that $s^n\in U(s)$. Obviously, if there exists a subgroup $H(e)$
with a neutral element $e$ in $S$, then $s\in H(e)$ is topologically
periodic if and only if for every open neighbourhood $U(e)$ of $e$
in $S$ there exists a positive integer $n$ such that $s^n\in U(e)$.

The bicyclic semigroup ${\mathscr{C}}(p,q)$ is the semigroup with
the identity $1$ generated by elements $p$ and $q$ subject only to
the condition $pq=1$. The distinct elements of ${\mathscr{C}}(p,q)$
are exhibited in the following useful array:
\begin{equation*}
\begin{array}{ccccc}
  1      & p      & p^2    & p^3    & \cdots \\
  q      & qp     & qp^2   & qp^3   & \cdots \\
  q^2    & q^2p   & q^2p^2 & q^2p^3 & \cdots \\
  q^3    & q^3p   & q^3p^2 & q^3p^3 & \cdots \\
  \vdots & \vdots & \vdots & \vdots & \ddots \\
\end{array}
\end{equation*}
The bicyclic semigroup is bisimple and every one of its congruences
is either trivial or a group congruence. Moreover, every
non-annihilating homomorphism $h$ of the bicyclic semigroup is
either an isomorphism or the image of ${\mathscr{C}}(p,q)$ under $h$
is a cyclic group~(see \cite[Corollary~1.32]{CP}). The bicyclic
semigroup plays an important role in algebraic theory of semigroups
and in the theory of topological semigroups. For example the
well-known Andersen's result~\cite{Andersen} states that a
($0$--)simple semigroup is completely ($0$--)simple if and only if
it does not contain the bicyclic semigroup. The bicyclic semigroup
admits only the discrete semigroup topology and a topological
semigroup $S$ can contain the bicyclic semigroup ${\mathscr C}(p,q)$
as a dense subsemigroup only as an open
subset~\cite{EberhartSelden1969}. Also Bertman and West in
\cite{BertmanWest1976} proved that the bicyclic semigroup as a
Hausdorff semitopological semigroup admits only the discrete
topology. The problem of an embedding of the bicycle semigroup into
compact-like topological semigroups solved in the papers \cite{AHK,
BanakhDimitrovaGutik2009, BanakhDimitrovaGutik2010, GutikRepovs2007,
HildebrantKoch1988} and the closure of the bicycle semigroup in
topological semigroups studied in~\cite{EberhartSelden1969}.

Let $\mathbb{Z}$ be the additive group of integers. On the Cartesian
product $\mathscr{C}_{\mathbb{Z}}=\mathbb{Z}\times\mathbb{Z}$ we
define the semigroup operation as follows:
\begin{equation}\label{f1}
    (a,b)\cdot(c,d)=
\left\{
  \begin{array}{ll}
    (a-b+c,d), & \hbox{if }~b<c; \\
    (a,d),     & \hbox{if }~b=c; \\
    (a,d+b-c), & \hbox{if }~b>c,\\
  \end{array}
\right.
\end{equation}
for $a,b,c,d\in\mathbb{Z}$. The set $\mathscr{C}_{\mathbb{Z}}$ with
such defined operation is called the \emph{extended bicycle
semigroup} \cite{Warne1968}.

In this paper we study the semigroup $\mathscr{C}_{\mathbb{Z}}$. We
describe main algebraic properties of the semigroup
$\mathscr{C}_{\mathbb{Z}}$ and prove that every non-trivial
congruence $\mathfrak{C}$ on the semigroup
$\mathscr{C}_{\mathbb{Z}}$ is a group congruence, and moreover the
quotient semigroup $\mathscr{C}_{\mathbb{Z}}/\mathfrak{C}$ is
isomorphic to a cyclic group. Also we show that the semigroup
$\mathscr{C}_{\mathbb{Z}}$ as a Hausdorff semitopological semigroup
admits only the discrete topology. Next we study the closure
$\operatorname{cl}_T\left(\mathscr{C}_{\mathbb{Z}}\right)$ of the
semigroup $\mathscr{C}_{\mathbb{Z}}$ in a topological semigroup $T$.
We show that the non-empty remainder of $\mathscr{C}_{\mathbb{Z}}$
in a topological inverse semigroup $T$ consists of a group of units
$H(1_T)$ of $T$ and a two-sided ideal $I$ of $T$ in the case when
$H(1_T)\neq\varnothing$ and $I\neq\varnothing$. In the case when $T$
is a locally compact topological inverse semigroup and
$I\neq\varnothing$ we prove that an ideal $I$ is topologically
isomorphic to the discrete additive group of integers and describe
the topology on the subsemigroup $\mathscr{C}_{\mathbb{Z}}\cup I$.
Also we show that if the group of units $H(1_T)$ of the semigroup
$T$ is non-empty, then $H(1_T)$ is either singleton or $H(1_T)$ is
topologically isomorphic to the discrete additive group of integers.

\section{Algebraic properties of the semigroup
$\mathscr{C}_{\mathbb{Z}}$}

\begin{proposition}\label{proposition-1} The following statements
hold:
\begin{itemize}
  \item[$(i)$] $E(\mathscr{C}_{\mathbb{Z}})=\{(a,a)\mid
   a\in\mathbb{Z}\}$, and $(a,a)\leqslant(b,b)$ in
   $E(\mathscr{C}_{\mathbb{Z}})$ if and only if $a\geqslant b$
   in~$\mathbb{Z}$, and hence $E(\mathscr{C}_{\mathbb{Z}})$ is
   isomorphic to the linearly ordered semilattice
   $(\mathbb{Z},\max)$;

  \item[$(ii)$] $\mathscr{C}_{\mathbb{Z}}$ is an inverse semigroup,
   and the elements $(a,b)$ and $(b,a)$ are inverse in
   $\mathscr{C}_{\mathbb{Z}}$;

  \item[$(iii)$] for any idempotents
   $e,f\in\mathscr{C}_{\mathbb{Z}}$ there exists
   $x\in\mathscr{C}_{\mathbb{Z}}$ such that $x\cdot x^{-1}=e$ and
   $x^{-1}\cdot x=f$;

  \item[$(iv)$] elements $(a,b)$ and $(c,d)$ of the semigroup
   $\mathscr{C}_{\mathbb{Z}}$ are:
   \begin{itemize}
     \item[$(a)$] $\mathscr{R}$-equivalent if and only if $a=c$;
     \item[$(b)$] $\mathscr{L}$-equivalent if and only if $b=d$;
     \item[$(c)$] $\mathscr{H}$-equivalent if and only if $a=c$ and
      $b=d$;
     \item[$(d)$] $\mathscr{D}$-equivalent for all
      $a,b,c,d\in\mathbb{Z}$;
     \item[$(e)$] $\mathscr{J}$-equivalent for all
      $a,b,c,d\in\mathbb{Z}$;
   \end{itemize}

  \item[$(v)$] $\mathscr{C}_{\mathbb{Z}}$ is a bisimple semigroup
   and hence it is simple;

  \item[$(vi)$] if $(a,b)\cdot(c,d)=(x,y)$ in
   $\mathscr{C}_{\mathbb{Z}}$ then $x-y=a-b+c-d$.

  \item[$(vii)$] every maximal subgroup of $\mathscr{C}_{\mathbb{Z}}$
   is trivial.

  \item[$(viii)$] for every integer $n$ the subsemigroup
   $\mathscr{C}_{\mathbb{Z}}[n]=\{(a,b)\mid a\geqslant n\;\&\;
   b\geqslant n\}$ of $\mathscr{C}_{\mathbb{Z}}$ is isomorphic to
   the bicyclic semigroup ${\mathscr{C}}(p,q)$, and moreover an
   isomorphism $h\colon\mathscr{C}_{\mathbb{Z}}[n]\rightarrow
   {\mathscr{C}}(p,q)$ is defined by the formula
   $\left((a,b)\right)h=q^{a-n}p^{b-n}$;

  \item[$(ix)$] $\mathscr{L\!I}_{\mathscr{C}_{\mathbb{Z}}}=
   \{\mathscr{L}^a\mid a\in\mathbb{Z}\}$,
   where $\mathscr{L}^a=\{(x,y)\in\mathscr{C}_{\mathbb{Z}}\mid
   y\geqslant a\}$, is the family of
   all left ideals of the semigroup $\mathscr{C}_{\mathbb{Z}}$;

  \item[$(x)$] $\mathscr{R\!I}_{\mathscr{C}_{\mathbb{Z}}}=
   \{\mathscr{R}^a\mid a\in\mathbb{Z}\}$,
   where $\mathscr{R}^a=\{(x,y)\in\mathscr{C}_{\mathbb{Z}}\mid
   x\geqslant a\}$, is the family of
   all right ideals of the semigroup $\mathscr{C}_{\mathbb{Z}}$.
\end{itemize}
\end{proposition}

\begin{proof}
The proofs of statements $(i)$, $(ii)$, $(iii)$, $(iv)$, $(vi)$,
$(vii)$ and $(viii)$ are trivial. Statement $(v)$ follows from
statement $(iii)$ and Lemma~1.1 of \cite{Munn1966}.

Simple verifications (see: formula (\ref{f1})) show that
\begin{equation*}
    (a,b)\mathscr{C}_{\mathbb{Z}}=\left\{
    (x,y)\in\mathscr{C}_{\mathbb{Z}}\mid x\geqslant a\right\}
\qquad \hbox{ and } \qquad
    \mathscr{C}_{\mathbb{Z}}(a,b)=\left\{
    (x,y)\in\mathscr{C}_{\mathbb{Z}}\mid y\geqslant b\right\}
\end{equation*}
for every $(a,b)\in\mathscr{C}_{\mathbb{Z}}$. This completes the
proof of statements $(ix)$ and $(x)$.
\end{proof}

\begin{proposition}\label{proposition-2}
Every non-trivial congruence $\mathfrak{C}$ on the semigroup
$\mathscr{C}_{\mathbb{Z}}$ is a group congruence, and moreover the
quotient semigroup $\mathscr{C}_{\mathbb{Z}}/\mathfrak{C}$ is
isomorphic to a cyclic group.
\end{proposition}

\begin{proof}
First we shall show that if two distinct idempotents $(a,a)$ and
$(b,b)$ of $\mathscr{C}_{\mathbb{Z}}$ are $\mathfrak{C}$-equivalent
then the quotient semigroup $\mathscr{C}_{\mathbb{Z}}/\mathfrak{C}$
is a group. Without loss of generality we can assume that
$(a,a)\leqslant(b,b)$, i.e., $a\geqslant b$ in $\mathbb{Z}$. Then we
have that
\begin{align*}
    (a,b)\cdot(b,b)\cdot(b,a)=\,&(a,a);\\
    (a,b)\cdot(a,a)\cdot(b,a)=\,&\left(a+(a-b),a+(a-b)\right);\\
    (a,b)\cdot\left(a+(a-b),a+(a-b)\right)\cdot(b,a)=\,
          &\left(a+2(a-b),a+2(a-b)\right);\\
  \cdots \qquad\cdots \qquad \cdots \qquad\qquad &
  \qquad \cdots \qquad \cdots\\
    (a,b)\cdot\left(a+j(a-b),a+j(a-b)\right)\cdot(b,a)=\,
          &\left(a+(j+1)(a-b),a+(j+1)(a-b)\right);\\
  \cdots \qquad\cdots \qquad \cdots \qquad\qquad &
  \qquad \cdots \qquad \cdots
\end{align*}
This implies that for every non-negative integers $i$ and $j$ we
have that
\begin{equation*}
\left(a+i(a-b),a+i(a-b)\right)\mathfrak{C}
\left(a+j(a-b),a+j(a-b)\right).
\end{equation*}

If $b\geqslant k$ in $\mathbb{Z}$ for some integer $k$, then by
Proposition~\ref{proposition-1}$(viii)$ we get that any two distinct
idempotents of the subsemigroup $\mathscr{C}_{\mathbb{N}}[k]$ of
$\mathscr{C}_{\mathbb{Z}}$ are $\mathfrak{C}$-equivalent and hence
Proposition~\ref{proposition-1}$(viii)$ and Corollary~1.32 from
\cite{CP} imply that for every integer $n$ all idempotents of the
subsemigroup $\mathscr{C}_{\mathbb{N}}[n]$ are
$\mathfrak{C}$-equivalent. This implies that all idempotents of the
subsemigroup $\mathscr{C}_{\mathbb{N}}[n]$ are
$\mathfrak{C}$-equivalent. Since the semigroup
$\mathscr{C}_{\mathbb{Z}}$ is inverse we conclude that the quotient
semigroup $\mathscr{C}_{\mathbb{Z}}/\mathfrak{C}$ contains only one
idempotent and hence by Lemma II.1.10 from \cite{Petrich1984} the
semigroup $\mathscr{C}_{\mathbb{Z}}/\mathfrak{C}$ is a group.

Suppose that two distinct elements $(a,b)$ and $(c,d)$ of the
semigroup $\mathscr{C}_{\mathbb{Z}}$ are $\mathfrak{C}$-equivalent.
Since $\mathscr{C}_{\mathbb{Z}}$ is an inverse semigroup,
Lemma~III.1.1 from \cite{Petrich1984} implies that
$(a,a)\mathfrak{C}(c,c)$ and $(b,b)\mathfrak{C}(d,d)$. Since
$(a,b)\neq(c,d)$ we have that either $(a,a)\neq(c,c)$ or
$(b,b)\neq(d,d)$, and hence by the first part of the proof we get
that all idempotents of the semigroup $\mathscr{C}_{\mathbb{Z}}$ are
$\mathfrak{C}$-equivalent.

Next we shall show that if $\mathfrak{C}_{mg}$ be a least group
congruence on the semigroup $\mathscr{C}_{\mathbb{Z}}$, then the
quotient semigroup $\mathscr{C}_{\mathbb{Z}}/{\mathfrak{C}_{mg}}$ is
isomorphic to the additive group of integers $\mathbb{Z}$.

By Proposition~\ref{proposition-1}$(i)$ and Lemma~III.5.2 from
\cite{Petrich1984} we have that elements $(a,b)$ and $(c,d)$ are
$\mathfrak{C}_{mg}$-equivalent in $\mathscr{C}_{\mathbb{Z}}$ if and
only if there exists an integer $n$ such that
$(a,b)\cdot(n,n)=(c,d)\cdot(n,n)$. Then
Proposition~\ref{proposition-1}$(i)$ implies that
$(a,b)\cdot(g,g)=(c,d)\cdot(g,g)$ for any integer $g$ such that
$g\geqslant n$ in $\mathbb{Z}$. If $g\geqslant b$ and $g\geqslant d$
in $\mathbb{Z}$, then the semigroup operation in
$\mathscr{C}_{\mathbb{Z}}$ implies that $(a,b)\cdot(g,g)=(g-b+a,g)$
and $(c,d)\cdot(g,g)=(g-d+c,g)$, and since $\mathbb{Z}$ is the
additive group of integers we get that $a-b=c-d$. Converse, suppose
that $(a,b)$ and $(c,d)$ are elements of the semigroup
$\mathscr{C}_{\mathbb{Z}}$ such that $a-b=c-d$. Then for any element
$g\in\mathbb{Z}$ such that $g\geqslant b$ and $g\geqslant d$ in
$\mathbb{Z}$ we have that $(a,b)\cdot(g,g)=(g-b+a,g)$ and
$(c,d)\cdot(g,g)=(g-d+c,g)$, and since $a-b=c-d$ we get that
$(a,b)\mathfrak{C}_{mg}(c,d)$. Therefore,
$(a,b)\mathfrak{C}_{mg}(c,d)$ in $\mathscr{C}_{\mathbb{Z}}$ if and
only if $a-b=c-d$.

We determine a map $\mathfrak{f}\colon\mathscr{C}_{\mathbb{Z}}
\rightarrow \mathbb{Z}$ by the formula
$\left((a,b)\right)\mathfrak{f}=a-b$, for $a,b\in \mathbb{Z}$.
Proposition~\ref{proposition-1}$(vi)$ implies that such defined map
$\mathfrak{f}\colon\mathscr{C}_{\mathbb{Z}} \rightarrow \mathbb{Z}$
is a homomorphism. Then we have that $(a,b)\mathfrak{C}_{mg}(c,d)$
if and only if
$\left((a,b)\right)\mathfrak{f}=\left((c,d)\right)\mathfrak{f}$, for
$(a,b),(c,d)\in\mathscr{C}_{\mathbb{Z}}$, and hence the homomorphism
$\mathfrak{f}$ generates the least group congruence
$\mathfrak{C}_{mg}$ on the semigroup $\mathscr{C}_{\mathbb{Z}}$.

If $\mathfrak{c}$ is any congruence on the semigroup
$\mathscr{C}_{\mathbb{Z}}$ then the mapping
$\mathfrak{c}\mapsto\mathfrak{c}\vee\mathfrak{C}_{mg}$ maps the
congruence $\mathfrak{c}$ onto a group congruence
$\mathfrak{c}\vee\mathfrak{C}_{mg}$, where $\mathfrak{C}_{mg}$ is
the least group congruence on the semigroup
$\mathscr{C}_{\mathbb{Z}}$ (cf. \cite[Section~III]{Petrich1984}).
Therefore every homomorphic image of the semigroup
$\mathscr{C}_{\mathbb{Z}}$ is a homomorphic image of the quotient
semigroup $\mathscr{C}_{\mathbb{Z}}/\mathfrak{C}$, i.e., it is a
homomorphic image of the additive group of integers $\mathbb{Z}$.
This completes the proof of the theorem.
\end{proof}


\section{The semigroup $\mathscr{C}_{\mathbb{Z}}$: topologizations
and closures of $\mathscr{C}_{\mathbb{Z}}$ in topological
semigroups}

\begin{theorem}\label{theorem-3}
Every Hausdorff topology $\tau$ on the semigroup
$\mathscr{C}_{\mathbb{Z}}$ such that
$(\mathscr{C}_{\mathbb{Z}},\tau)$ is a semitopological semigroup is
discrete, and hence $\mathscr{C}_{\mathbb{Z}}$ is a discrete
subspace of any semitopological semigroup which contains
$\mathscr{C}_{\mathbb{Z}}$ as a subsemigroup.
\end{theorem}

\begin{proof}
We fix an arbitrary idempotent $(a,a)$ of the semigroup
$\mathscr{C}_{\mathbb{Z}}$ and suppose that $(a,a)$ is a
non-isolated point of the topological space
$(\mathscr{C}_{\mathbb{Z}},\tau)$. Since the maps
$\lambda_{(a,a)}\colon
\mathscr{C}_{\mathbb{Z}}\rightarrow\mathscr{C}_{\mathbb{Z}}$ and
$\rho_{(a,a)}\colon \mathscr{C}_{\mathbb{Z}}\rightarrow
\mathscr{C}_{\mathbb{Z}}$ defined by the formulae
$\left((x,y)\right)\lambda_{(a,a)}=(a,a)\cdot(x,y)$ and
$\left((x,y)\right)\rho_{(a,a)}=(x,y)\cdot(a,a)$ are continuous
retractions we conclude that $(a,a)\mathscr{C}_{\mathbb{Z}}$ and
$\mathscr{C}_{\mathbb{Z}}(a,a)$ are closed subsets in the
topological space $(\mathscr{C}_{\mathbb{Z}},\tau)$. We put
\begin{equation*}
    \textsf{DL}_{(a,a)}\left[(a,a)\right]=
    \left\{(x,y)\in\mathscr{C}_{\mathbb{Z}}
    \mid (x,y)\cdot(a,a)=(a,a)\right\}.
\end{equation*}
Simple verifications show that
\begin{equation*}
    \textsf{DL}_{(a,a)}\left[(a,a)\right]=
    \left\{(x,x)\in\mathscr{C}_{\mathbb{Z}}
    \mid x\leqslant a \; \hbox{ in } \; \mathbb{Z}\right\},
\end{equation*}
and since right translations are continuous maps in
$(\mathscr{C}_{\mathbb{Z}},\tau)$ we get that
$\textsf{DL}_{(a,a)}\left[(a,a)\right]$ is a closed subset of the
topological space $(\mathscr{C}_{\mathbb{Z}},\tau)$. Then there
exists an open neighbourhood $W_{(a,a)}$ of the point $(a,a)$ in the
topological space $(\mathscr{C}_{\mathbb{Z}},\tau)$ such that
\begin{equation*}
    W_{(a,a)}\subseteq\mathscr{C}_{\mathbb{Z}}\setminus
    \big((a+1,a+1)\mathscr{C}_{\mathbb{Z}}\cup
    \mathscr{C}_{\mathbb{Z}}(a+1,a+1)\cup
    \textsf{DL}_{(a-1,a-1)}(a-1,a-1)\big).
\end{equation*}
Since $(\mathscr{C}_{\mathbb{Z}},\tau)$ is a semitopological
semigroup we conclude that there exists an open neighbourhood
$V_{(a,a)}$ of the idempotent $(a,a)$ in the topological space
$(\mathscr{C}_{\mathbb{Z}},\tau)$ such that the following conditions
hold:
\begin{equation*}
    V_{(a,a)}\subseteq W_{(a,a)}, \qquad
    (a,a)\cdot V_{(a,a)}\subseteq W_{(a,a)} \qquad \hbox{and} \qquad
    V_{(a,a)}\cdot(a,a) \subseteq W_{(a,a)}.
\end{equation*}
Hence at least one of the following conditions holds:
\begin{itemize}
  \item[$(a)$] the neighbourhood $V_{(a,a)}$ contains infinitely
   many points $(x,y)\in\mathscr{C}_{\mathbb{Z}}$ such that
   $x<y\leqslant a$; \; or

  \item[$(b)$] the neighbourhood $V_{(a,a)}$ contains infinitely
   many points $(x,y)\in\mathscr{C}_{\mathbb{Z}}$ such that
   $y<x\leqslant a$.
\end{itemize}
In case $(a)$ we have that
\begin{equation*}
    (a,a)\cdot(x,y)=\big(a,a+(y-x)\big)\notin W_{(a,a)},
\end{equation*}
because $y-x\geqslant 1$, and in case $(b)$ we have that
\begin{equation*}
    (x,y)\cdot(a,a)=\big(a+(x-y),a\big)\notin W_{(a,a)},
\end{equation*}
because $x-y\geqslant 1$, a contradiction. The obtained
contradiction implies that the set $V_{(a,a)}$ is singleton, and
hence the idempotent $(a,a)$ is an isolated point of the topological
space $(\mathscr{C}_{\mathbb{Z}},\tau)$.

Let $(a,b)$ be an arbitrary element of the semigroup
$\mathscr{C}_{\mathbb{Z}}$ and suppose that $(a,b)$ is a
non-isolated point of the topological space
$(\mathscr{C}_{\mathbb{Z}},\tau)$. Since all right translations are
continuous maps in $(\mathscr{C}_{\mathbb{Z}},\tau)$ and every
idempotent $(a,a)$ of $\mathscr{C}_{\mathbb{Z}}$ is an isolated
point of the topological space $(\mathscr{C}_{\mathbb{Z}},\tau)$ we
conclude that
\begin{equation*}
    \textsf{DL}_{(b,a)}\left[(a,a)\right]=
    \big\{(x,y)\in\mathscr{C}_{\mathbb{Z}}
    \mid (x,y)\cdot(b,a)=(a,a)\big\}
\end{equation*}
is a closed-and-open subset of the topological space
$(\mathscr{C}_{\mathbb{Z}},\tau)$. Simple verifications show that
\begin{equation*}
    \textsf{DL}_{(b,a)}\left[(a,a)\right]=
    \big\{(x,y)\in\mathscr{C}_{\mathbb{Z}}
    \mid x-y=a-b \; \hbox{ and } \; x\leqslant a\big\}.
\end{equation*}
Then we have that
\begin{equation*}
    \big\{(a,b)\big\}=\textsf{DL}_{(b,a)}\left[(a,a)\right]
    \setminus\textsf{DL}_{(b-1,a-1)}\left[(a-1,a-1)\right],
\end{equation*}
and hence $(a,b)$ is an isolated point of the topological space
$(\mathscr{C}_{\mathbb{Z}},\tau)$. This completes the proof of the
theorem.
\end{proof}

Theorem~\ref{theorem-3} implies the following:

\begin{corollary}\label{corollary-4}
Every Hausdorff semigroup topology $\tau$ on
$\mathscr{C}_{\mathbb{Z}}$ is discrete, and hence
$\mathscr{C}_{\mathbb{Z}}$ is a discrete subspace of any topological
semigroup which contains $\mathscr{C}_{\mathbb{Z}}$ as a
subsemigroup.
\end{corollary}

Since every discrete topological space is locally compact,
Theorem~\ref{theorem-3} and Theorem~3.3.9 from \cite{Engelking1989}
imply the following:

\begin{corollary}\label{corollary-5}
Let $T$ be a semitopological semigroup which contains
$\mathscr{C}_{\mathbb{Z}}$ as a subsemigroup. Then
$\mathscr{C}_{\mathbb{Z}}$ is an open subsemigroup of $T$.
\end{corollary}

\begin{lemma}\label{lemma-6}
Let $T$ be a Hausdorff semitopological semigroup which contains
$\mathscr{C}_{\mathbb{Z}}$ as a dense subsemigroup. Let $f\in
T\setminus\mathscr{C}_{\mathbb{Z}}$ be an idempotent of the
semigroup $T$ which satisfies the property: there exists an
idempotent $(n,n)\in\mathscr{C}_{\mathbb{Z}}$, $n\in\mathbb{Z}$,
such that $(n,n)\leqslant f$. Then the following statements hold:
\begin{itemize}
  \item[$(i)$] there exists an open neighbourhood $U(f)$ of $f$
   in $T$ such that $U(f)\cap\mathscr{C}_{\mathbb{Z}}\subseteq
   E(\mathscr{C}_{\mathbb{Z}})$;

  \item[$(ii)$] $f$ is the unit of $T$.
\end{itemize}
\end{lemma}

\begin{proof}
$(i)$ Let $W(f)$ be an arbitrary open neighbourhood of the
idempotent $f$ in $T$. We fix an arbitrary element
$(n,n)\in\mathscr{C}_{\mathbb{Z}}$, $n\in\mathbb{Z}$. By
Corollary~\ref{corollary-5} the element $(n,n)$ is an isolated point
in $T$, and since $T$ is a semitopological semigroup we have that
there exists an open neighbourhood $U(f)$ of $f$ in $T$ such that
\begin{equation*}
    U(f)\subseteq W(f), \qquad U(f)\cdot\{(n,n)\}=\{(n,n)\}
    \qquad \hbox{ and } \qquad \{(n,n)\}\cdot U(f)=\{(n,n)\}.
\end{equation*}
If the set $U(f)$ contains a non-idempotent element
$(x,y)\in\mathscr{C}_{\mathbb{Z}}$, then
Proposition~\ref{proposition-1}$(vi)$ implies that $(x,y)\cdot(n,n),
(n,n)\cdot(x,y)\notin E(\mathscr{C}_{\mathbb{Z}})$, a contradiction.
The obtained contradiction implies the statement of the assertion.

$(ii)$ First we show that $f\cdot(k,l)=(k,l)\cdot f=(k,l)$ for every
$(k,l)\in\mathscr{C}_{\mathbb{Z}}$.

Suppose the contrary: there exists an element
$(k,l)\in\mathscr{C}_{\mathbb{Z}}$ such that
$x=f\cdot(k,l)\neq(k,l)$ for some $x\in T$. Let $U(x)$ be an open
neighbourhood of $x$ in $T$ such that $(k,l)\notin U(x)$. Since $T$
is a semitopological semigroup we get that there exists an open
neighbourhood $V(f)$ of $f$ in $T$ such that
$V(f)\cdot\{(k,l)\}\subseteq U(x)$. Again, since for an arbitrary
integer $a$ the maps $\lambda_{(a,a)}\colon
\mathscr{C}_{\mathbb{Z}}\rightarrow\mathscr{C}_{\mathbb{Z}}$ and
$\rho_{(a,a)}\colon \mathscr{C}_{\mathbb{Z}}\rightarrow
\mathscr{C}_{\mathbb{Z}}$ defined by the formulae
$\left((x,y)\right)\lambda_{(a,a)}=(a,a)\cdot(x,y)$ and
$\left((x,y)\right)\rho_{(a,a)}=(x,y)\cdot(a,a)$ are continuous
retractions we conclude that statement $(i)$ implies that there
exists an open neighbourhood $W(f)$ of $f$ in $T$ such that
$W(f)\subseteq V(f)$, $W(f)\cap\mathscr{C}_{\mathbb{Z}}\subseteq
E(\mathscr{C}_{\mathbb{Z}})$ and the following condition holds:
\begin{equation*}
    (p,p)\in W(f)\cap\mathscr{C}_{\mathbb{Z}} \qquad \hbox{ if and
    only if } \qquad p\geqslant k.
\end{equation*}
Then $(p,p)\cdot(k,l)=(k,l)\notin U(x)$ for every $(p,p)\in
W(f)\cap\mathscr{C}_{\mathbb{Z}}$, a contradiction. The obtained
contradiction implies that $f\cdot(k,l)=(k,l)$ for every
$(k,l)\in\mathscr{C}_{\mathbb{Z}}$. Similar arguments show that
$(k,l)\cdot f=(k,l)$ for every $(k,l)\in\mathscr{C}_{\mathbb{Z}}$.

Next we show that $f\cdot x=x\cdot f=x$ for every $x\in
T\setminus\mathscr{C}_{\mathbb{Z}}$. Suppose the contrary: there
exists an element $x\in T\setminus\mathscr{C}_{\mathbb{Z}}$ such
that $y=f\cdot x\neq x$ for some $y\in T$. Let $U(x)$ and $U(y)$ be
open neighbourhoods of $x$ and $y$ in $T$, respectively, such that
$U(x)\cap U(y)=\varnothing$. Since $T$ is a semitopological
semigroup we get that there exists an open neighbourhood $V(x)$ of
$x$ in $T$ such that $V(x)\subseteq U(x)$ and $f\cdot V(x)\subseteq
U(y)$. Again, since $x\in T\setminus\mathscr{C}_{\mathbb{Z}}$ we
have that the set $V(x)\cap\mathscr{C}_{\mathbb{Z}}$ is infinite,
and the previous part of the proof of the statement implies that
$f\cdot \left(V(x)\cap\mathscr{C}_{\mathbb{Z}}\right)\subseteq
\left(V(x)\cap\mathscr{C}_{\mathbb{Z}}\right)$. But we have that
$V(x)\cap U(y)=\varnothing$, a contradiction. The obtained
contradiction implies the equality $f\cdot x=x$. Similar arguments
show that $x\cdot f=x$  for every $x\in
T\setminus\mathscr{C}_{\mathbb{Z}}$.
\end{proof}

\begin{remark}\label{remark-7}
We observe that the assertion $(i)$ of Lemma~\ref{lemma-6} holds for
right-topological and left-topological monoids.
\end{remark}

\begin{lemma}\label{lemma-8}
Let $T$ be a Hausdorff topological monoid with the unit $1_T$ which
contains $\mathscr{C}_{\mathbb{Z}}$ as a dense subsemigroup. Then
the following assertions hold:
\begin{itemize}
  \item[$(i)$] there exists an open neighbourhood $U(1_T)$ of the
   unit $1_T$ in $T$ such that
   $U(1_T)\cap\mathscr{C}_{\mathbb{Z}}\subseteq
   E(\mathscr{C}_{\mathbb{Z}})$;

\end{itemize}
and if the group of units $H(1_T)$ of $T$ is non-singleton, then:
\begin{itemize}
  \item[$(ii)$] for every $x\in H(1_T)$ there exists an open
   neighbourhood $U(x)$ in $T$ such that $a-b=c-d$ for all
   $(a,b),(c,d)\in U(x)\cap\mathscr{C}_{\mathbb{Z}}$;

  \item[$(iii)$] for distinct $x,y\in H(1_T)$ there exist open
   neighbourhoods $U(x)$ and
   $U(y)$ of $x$ and $y$ in $T$, respectively, such that $a-b\neq
   c-d$ for every $(a,b)\in U(x)\cap\mathscr{C}_{\mathbb{Z}}$ and
   for every $(c,d)\in U(y)\cap\mathscr{C}_{\mathbb{Z}}$;

  \item[$(iv)$] the group $H(1_T)$ is torsion free;

  \item[$(v)$] the group of units $H(1_T)$ of $T$ is a discrete
   subgroup in $T$;

  \item[$(vi)$] the group of units $H(1_T)$ of $T$ is isomorphic to
   the infinite cyclic group;

  \item[$(vii)$] every non-identity element of the group of units
   $H(1_T)$ in the semigroup $T$ is not topologically periodic.
\end{itemize}
\end{lemma}

\begin{proof}
Statement $(i)$ follows from Lemma~\ref{lemma-6}$(i)$.

$(ii)$ In the case $H(1_T)=\{1_T\}$ statement $(i)$ implies our
assertion. Hence we suppose that $H(1_T)\neq\{1_T\}$ and let $x\in
H(1_T)\setminus\{1_T\}$. By statement $(i)$ there exists an open
neighbourhood $U(1_T)$ of the unit $1_T$ in $T$ such that
$U(1_T)\cap\mathscr{C}_{\mathbb{Z}}\subseteq
E(\mathscr{C}_{\mathbb{Z}})$. Then the continuity of the semigroup
operation in $T$ implies that there exist open neighbourhoods $U(x)$
and $U(x^{-1})$ in the topological space $T$ of $x$ and the inverse
element $x^{-1}$ of $x$ in $H(1_T)$, respectively, such that
\begin{equation*}
    U(x)\cdot U(x^{-1})\subseteq U(1_T) \qquad \hbox{ and } \qquad
    U(x^{-1})\cdot U(x)\subseteq U(1_T).
\end{equation*}
Since $U(1_T)\cap\mathscr{C}_{\mathbb{Z}}\subseteq
E(\mathscr{C}_{\mathbb{Z}})$ we have that
Proposition~\ref{proposition-1}$(vi)$ implies that $a-b+u-v=c-d+u-v$
for all $(a,b),(c,d)\in U(x)\cap\mathscr{C}_{\mathbb{Z}}$ and some
$(u,v)\in U(x^{-1})\cap\mathscr{C}_{\mathbb{Z}}$, and hence
$a-b=c-d$.

$(iii)$ Suppose the contrary: there exist distinct $x,y\in H(1_T)$
and for all open neighbourhoods $U(x)$ and $U(y)$ of $x$ and $y$ in
$T$, respectively, there are $(a,b)\in
U(x)\cap\mathscr{C}_{\mathbb{Z}}$ and $(c,d)\in
U(y)\cap\mathscr{C}_{\mathbb{Z}}$ such that $a-b=c-d$. The
Hausdorffness of $T$ implies that without loss of generality we can
assume that $U(x)\cap U(y)=\varnothing$. Then statement $(i)$ and
the continuity of the semigroup operation in $T$ imply that there
exist open neighbourhoods $V(1_T)$, $V(x)$ and $V(y)$ of $1_T$, $x$
and $y$ in $T$, respectively, such that
\begin{equation*}
\begin{split}
  V(1_T)\cap\mathscr{C}_{\mathbb{Z}}\subseteq
    E(\mathscr{C}_{\mathbb{Z}}), & \; V(x)\subseteq U(x),
    \; V(y)\subseteq U(y), \;
    V(1_T)\cdot V(x)\subseteq U(x) \; \\
    & \hbox{ and } \;
    V(1_T)\cdot V(y)\subseteq U(y).
\end{split}
\end{equation*}
Since by Theorem~1.7 from \cite[Vol.~1]{CHK} the sets $(a,a)T$ and
$T(a,a)$ are closed in $T$ for every idempotent $(a,a)\in
\mathscr{C}_{\mathbb{Z}}$ and both neighbourhoods $V(x)$ and $V(y)$
contain infinitely many elements of the semigroup
$\mathscr{C}_{\mathbb{Z}}$ we conclude that for every $(p,p)\in
V(1_T)\cap\mathscr{C}_{\mathbb{Z}}$ there exist $(k,l)\in V(x)\cap
\mathscr{C}_{\mathbb{Z}}$ and $(m,n)\in V(y)\cap
\mathscr{C}_{\mathbb{Z}}$ such that
\begin{equation*}
    p>k>m, \qquad p>l>n \qquad \hbox{ and } \qquad k-l=m-n.
\end{equation*}
Then we get that
\begin{equation*}
    (p,p)\cdot(k,l)=(p,p+(l-k)) \qquad \hbox{ and } \qquad
    (p,p)\cdot(m,n)=(p,p+(n-m)),
\end{equation*}
a contradiction. The obtained contradiction implies our assertion.

$(iv)$ Suppose the contrary: there exist $x\in
H(1_T)\setminus\{1_T\}$ and a positive integer $n$ such that
$x^n=1_T$. Then by statement $(i)$ there exists an open
neighbourhood $U(1_T)$ of the unit $1_T$ in $T$ such that
$U(1_T)\cap\mathscr{C}_{\mathbb{Z}}\subseteq
E(\mathscr{C}_{\mathbb{Z}})$. The continuity of the semigroup
operation in $T$ and statement $(ii)$ imply that there exists an
open neighbourhood $V(x)$ of $x$ in $T$ such that $a-b=c-d$ for all
$(a,b),(c,d)\in V(x)\cap\mathscr{C}_{\mathbb{Z}}$ and
$\underbrace{V(x)\cdot \ldots\cdot V(x)}_{n\textrm{-times}}\subseteq
U(1_T)$. We fix an arbitrary element $(a,b)\in
V(x)\cap\mathscr{C}_{\mathbb{Z}}$. If $(a,b)^n=(x,y)$, then
Proposition~\ref{proposition-1}$(vi)$ implies that $x-y=n\cdot(a-b)$
and since $x\neq 1_T$ we get that $(x,y)\notin U(1_T)$, a
contradiction. The obtained contradiction implies statement $(iv)$.

$(v)$ Statement $(iv)$ implies that the group of units $H(1_T)$ is
infinite.

We fix an arbitrary $x\in H(1_T)$ and suppose that $x$ is not an
isolated point of $H(1_T)$. Then by statement $(ii)$ there exists an
open neighbourhood $U(x)$ in $T$ such that $a-b=c-d$ for all
$(a,b),(c,d)\in U(x)\cap\mathscr{C}_{\mathbb{Z}}$. Since the point
$x$ is not isolated in $H(1_T)$ we conclude that there exists $y\in
H(1_T)$ such that $y\in U(x)$. Hence the set $U(x)$ is an open
neighbourhood of $y$ in $T$. Statement $(iii)$ implies that there
exist open neighbourhoods $W(x)\subseteq U(x)$ and $W(y)\subseteq
U(x)$ of $x$ and $y$ in $T$, respectively, such that $a-b\neq c-d$
for every $(a,b)\in W(x)\cap\mathscr{C}_{\mathbb{Z}}$ and for every
$(c,d)\in W(y)\cap\mathscr{C}_{\mathbb{Z}}$. This contradicts the
choice of the neighbourhood $U(x)$. The obtained contradiction
implies that every $x\in H(1_T)$ is an isolated point of $H(1_T)$.

$(vi)$ Since the group of units $H(1_T)$ is not trivial, i.e., the
group $H(1_T)$ is non-singleton, we fix an arbitrary $x\in
H(1_T)\setminus\{1_T\}$. Then by statement $(iv)$ we have that
$x^n\neq 1_T$ for any positive integer $n$. Statement $(ii)$ implies
that there exists an open neighbourhood $U(x)$ in $T$ such that
$a-b=c-d$ for all $(a,b),(c,d)\in U(x)\cap\mathscr{C}_{\mathbb{Z}}$.
We define the map $\varphi\colon H(1_T)\rightarrow\mathbb{Z}$ by the
following way: $(x)\varphi=k$ if and only if $a-b=k$ for every
$(a,b)\in U(x)\cap\mathscr{C}_{\mathbb{Z}}$. Then statement $(iv)$
and Proposition~\ref{proposition-1}$(vi)$ imply that the map
$\varphi\colon H(1_T)\rightarrow\mathbb{Z}$ is an injective
homomorphism. Obviously that $\left(H(1_T)\right)\varphi$ is a
subgroup in the additive group of integers. We fix the least
positive integer $p\in\left(H(1_T)\right)\varphi$. Then the element
$p$ generates the subgroup $\left(H(1_T)\right)\varphi$ in the
additive group of integers $\mathbb{Z}$, and hence the group
$\left(H(1_T)\right)\varphi$ is cyclic.

$(vii)$ We fix an arbitrary element $x\in H(1_T)\setminus\{ 1_T\}$.
Suppose the contrary: $x$ is a topologically periodic element of
$S$. Then there exist open neighbourhoods $U(1_T)$ and $U(x)$ of
$1_T$ and $x$ in $T$, respectively, such that $U(1_T)\cap
U(x)=\varnothing$. Statements $(i)$ and $(iii)$ imply that without
loss of generality we can assume that $U(1_T)\cap
\mathscr{C}_{\mathbb{Z}}\subseteq E(\mathscr{C}_{\mathbb{Z}})$, and
$a-b=c-d\neq 0$ for all $(a,b),(c,d)\in U(x)\cap
\mathscr{C}_{\mathbb{Z}}$. Then the topologically periodicity of $x$
implies that there exists a positive integer $n$ such that $x^n\in
U(1_T)$. Since the semigroup operation in $T$ is continuous we
conclude that there exists an open neighbourhood $V(x)$ of $x$ in
$T$ such that $\underbrace{V(x)\cdot\ldots\cdot
V(x)}_{n\textrm{-times}}\subseteq U(1_T)$. We fix an arbitrary
element $(a,b)\in V(x)\cap\mathscr{C}_{\mathbb{Z}}$. Then we have
that $(a,b)^n\in U(1_T)\cap\mathscr{C}_{\mathbb{Z}}$ and hence
$n(a-b)=0$, a contradiction. The obtained contradiction implies
assertion $(vii)$.
\end{proof}

\begin{proposition}\label{proposition-9}
Let $G$ be non-trivial subgroup of the additive group of integers
$\mathbb{Z}$ and $n\in\mathbb{Z}$. Then the subsemigroup $H$ which
is generated by the set $\{n\}\cup G$ is a cyclic subgroup of
$\mathbb{Z}$.
\end{proposition}

\begin{proof}
Without loss of generality we can assume that
$n\in\mathbb{Z}\setminus G$ and $n>0$.

Since every subgroup of a cyclic group is cyclic (see
\cite[P.~47]{Kurosh1960}), we have that $G$ is a cyclic subgroup in
$\mathbb{Z}$. We fix a generating element $k$ of $G$ such that
$k>0$. Then we have that
\begin{equation*}
    (\underbrace{n+\cdots+n}_{(k-1)\textrm{-times}})-
    (\underbrace{k+\cdots+k}_{n\textrm{-times}})+n=0,
\end{equation*}
and hence we have that $-n\in H$. Since $\mathbb{Z}$ is a
commutative group we conclude that $H$ is a subgroup in
$\mathbb{Z}$, which is generated by elements $n$ and $k$, and hence
$H$ is a cyclic subgroup in $\mathbb{Z}$.
\end{proof}

\begin{proposition}\label{proposition-10}
Let $T$ be a Hausdorff topological monoid with the unit $1_T$ which
contains $\mathscr{C}_{\mathbb{Z}}$ as a dense subsemigroup. Then
the following assertions hold:
\begin{itemize}
  \item[$(i)$] if the set $L_{\mathscr{C}_{\mathbb{Z}}}=\left\{x\in
   T\setminus\mathscr{C}_{\mathbb{Z}}\mid \hbox{ there exists }
   y\in\mathscr{C}_{\mathbb{Z}}
   \hbox{ such that } x\cdot y\in\mathscr{C}_{\mathbb{Z}}\right\}$
   is non-empty, then $L_{\mathscr{C}_{\mathbb{Z}}}$ is a
   subsemigroup of $T$, and moreover if $a\in
   L_{\mathscr{C}_{\mathbb{Z}}}$, then there exists an
   open neighbourhood $U(a)$ of $a$ in $T$ such that
   $n_1-m_1=n_2-m_2$ for all $(n_1,m_1),(n_2,m_2)\in
   U(a)\cap\mathscr{C}_{\mathbb{Z}}$;

  \item[$(ii)$] if the set $R_{\mathscr{C}_{\mathbb{Z}}}=\left\{x\in
   T\setminus\mathscr{C}_{\mathbb{Z}}\mid \hbox{ there exists }
   y\in\mathscr{C}_{\mathbb{Z}} \hbox{ such that } y\cdot
   x\in\mathscr{C}_{\mathbb{Z}}\right\}$ is non-empty, then
   $R_{\mathscr{C}_{\mathbb{Z}}}$ is a subsemigroup of $T$, and
   moreover if $a\in R_{\mathscr{C}_{\mathbb{Z}}}$, then
   there exists an open neighbourhood $U(a)$ of $a$ in $T$ such that
   $n_1-m_1=n_2-m_2$ for all $(n_1,m_1),(n_2,m_2)\in
   U(a)\cap\mathscr{C}_{\mathbb{Z}}$;

  \item[$(iii)$] if the set $L_{\mathscr{C}_{\mathbb{Z}}}$ (resp.,
   $R_{\mathscr{C}_{\mathbb{Z}}}$) is non-empty, then for every
   $a\in L_{\mathscr{C}_{\mathbb{Z}}}$ (resp.,
   $a\in R_{\mathscr{C}_{\mathbb{Z}}}$) there exist an open
   neighbourhood $U(a)$ of $a$ in $T$ and an integer $n_a$ such that
   $p\leqslant n_a$ and $q\leqslant n_a$ for all
   $(p,q)\in U(a)\cap\mathscr{C}_{\mathbb{Z}}$;

  \item[$(iv)$] $L_{\mathscr{C}_{\mathbb{Z}}}=
   R_{\mathscr{C}_{\mathbb{Z}}}$;

  \item[$(v)$] ${\uparrow}\mathscr{C}_{\mathbb{Z}}=
   \mathscr{C}_{\mathbb{Z}}\cup L_{\mathscr{C}_{\mathbb{Z}}}$ is a
   subsemigroup of $T$ and $\mathscr{C}_{\mathbb{Z}}$ is a minimal
   ideal in ${\uparrow}\mathscr{C}_{\mathbb{Z}}$;

  \item[$(vi)$] if for an element $a\in T\setminus
   \mathscr{C}_{\mathbb{Z}}$ there is an open neighbourhood $U(a)$ of
   $a$ in $T$ and the following conditions hold:
   \begin{itemize}
    \item[$(a)$] $m_1-m_2=n_1-n_2$ for all $(m_1,n_1),(m_2,n_2)\in
     U(a)\cap\mathscr{C}_{\mathbb{Z}}$; \; and
    \item[$(b)$] there exists an integer $n_a$ such that $n\leqslant
     n_a$ and $m\leqslant n_a$ for every $(m,n)\in U(a)\cap
     \mathscr{C}_{\mathbb{Z}}$,
   \end{itemize}
   then $a\in L_{\mathscr{C}_{\mathbb{Z}}}$;

  \item[$(vii)$] if $I=T\setminus{\uparrow}\mathscr{C}_{\mathbb{Z}}
   \neq\varnothing$, then $I$ is an ideal of $T$;

  \item[$(viii)$] the set
\begin{equation*}
\begin{split}
  {\uparrow}(a,b)& = \{x\in T\mid x\cdot(b,b)=(a,b)\} \\
    & =\{x\in T\mid (a,a)\cdot x=(a,b)\} \\
    & = \{x\in T\mid (a,a)\cdot x\cdot(b,b)=(a,b)\}
\end{split}
\end{equation*}
 is closed-and-open in $T$ for every
   $(a,b)\in\mathscr{C}_{\mathbb{Z}}$;

  \item[$(ix)$] the set ${\uparrow}(a,b)\cap
   L_{\mathscr{C}_{\mathbb{Z}}}$ is either singleton or empty;

  \item[$(x)$] $L_{\mathscr{C}_{\mathbb{Z}}}$ is isomorphic to a
    submonoid of the additive group of integers $\mathbb{Z}$, and
    moreover if a maximal subgroup of $L_{\mathscr{C}_{\mathbb{Z}}}$
    is non-singleton, then $L_{\mathscr{C}_{\mathbb{Z}}}$ is
    isomorphic to the additive group of integers $\mathbb{Z}$;

  \item[$(xi)$] ${\uparrow}\mathscr{C}_{\mathbb{Z}}$ is an open
   subset in $T$, and hence if
   $I=T\setminus{\uparrow}\mathscr{C}_{\mathbb{Z}}\neq\varnothing$,
   then the ideal $I$ is a closed subset in $T$;

  \item[$(xii)$] if the semigroup $T$ contains a non-singleton group
   of units $H(1_T)$, then
   $H(1_T)=T\setminus(\mathscr{C}_{\mathbb{Z}}\cup I)$.
\end{itemize}
\end{proposition}

\begin{proof}
$(i)$ We observe that since $\mathscr{C}_{\mathbb{Z}}$ is an inverse
semigroup we conclude that $x\in L_{\mathscr{C}_{\mathbb{Z}}}$ if
and only if there exists an idempotent
$e\in\mathscr{C}_{\mathbb{Z}}$ such that $x\cdot
e\in\mathscr{C}_{\mathbb{Z}}$, for $x\in T$.

We fix an arbitrary $x\in L_{\mathscr{C}_{\mathbb{Z}}}$. Let $(n,n)$
be an idempotent in $\mathscr{C}_{\mathbb{Z}}$ such that
$(a,b)=x\cdot (n,n)\in\mathscr{C}_{\mathbb{Z}}$. Then by
Corollary~\ref{corollary-4} we have that $(n,n)$ and $(a,b)$ are
isolated points in $T$, and the continuity of the semigroup
operation in $T$ implies that there exists an open neighbourhood
$U(x)$ of $x$ in $T$ such that
\begin{equation*}
    U(x)\cdot\{(n,n)\}=\{(a,b)\}\in\mathscr{C}_{\mathbb{Z}}.
\end{equation*}
Then Proposition~\ref{proposition-1}$(vi)$ implies that $p-q=a-b$
for all $(p,q)\in U(x)\cap\mathscr{C}_{\mathbb{Z}}$. Also, since
\begin{equation}\label{f2}
    (p,q)(n,n)=
\left\{
  \begin{array}{ll}
    (p-q+n,n), & \hbox{if }~q\leqslant n; \\
    (p,q), & \hbox{if }~q\geqslant n
  \end{array}
\right.
\end{equation}
we have that $q\leqslant n=b$.

Suppose that $x,y\in L_{\mathscr{C}_{\mathbb{Z}}}$, and $(i,i)$ and
$(j,j)$ are idempotents in $\mathscr{C}_{\mathbb{Z}}$ such that
$x\cdot (i,i)=(k,l)\in \mathscr{C}_{\mathbb{Z}}$ and $y\cdot
(j,j)\in\mathscr{C}_{\mathbb{Z}}$, $i,j,k,l\in\mathbb{Z}$. We fix an
arbitrary integer $d$ such that $d\geqslant\max\{k,j\}$. Then we
have that
\begin{equation*}
\begin{split}
  (y\cdot x)\cdot \left((i,i)\cdot(l,k)\cdot(d,d)\right)& =
    y\cdot \left( x\cdot (i,i)\cdot(l,k)\cdot(d,d)\right)\\
    & =y\cdot \left( (k,l)\cdot(l,k)\cdot(d,d)\right) \\
    & = y\cdot \left( (k,k)\cdot(d,d)\right)\\
    & = y\cdot (d,d)\\
    & = y\cdot \left( (j,j)\cdot(d,d)\right)\\
    & = \left( y\cdot (j,j)\right)\cdot(d,d)\in
\mathscr{C}_{\mathbb{Z}}.
\end{split}
\end{equation*}
This implies that $L_{\mathscr{C}_{\mathbb{Z}}}$ is a subsemigroup
of $T$ and completes the proof of our assertion.

The proof of assertion $(ii)$ is similar to $(i)$.

Statement $(i)$ and formula (\ref{f2}) imply assertion $(iii)$. In
the case $a\in R_{\mathscr{C}_{\mathbb{Z}}}$ the proof is similar.

$(iv)$ Let be $L_{\mathscr{C}_{\mathbb{Z}}}\neq\varnothing$. We fix
an arbitrary element $a\in L_{\mathscr{C}_{\mathbb{Z}}}$. Then there
exists an idempotent $(i_a,i_a)\in\mathscr{C}_{\mathbb{Z}}$ such
that $a\cdot(i_a,i_a)=(i,j)\in\mathscr{C}_{\mathbb{Z}}$. Assertion
$(iii)$ implies that there exist an open neighbourhood $U(a)$ of $a$
in $T$ and an integer $n_a$ such that $n-m=i-j$, $n\leqslant n_a$
and $m\leqslant n_a$ for all $(n,m)\in U(a)\cap
\mathscr{C}_{\mathbb{Z}}$. Without loss of generality we can assume
that $i_a\geqslant n_a$.

We shall show that $(i_a,i_a)\cdot a\in\mathscr{C}_{\mathbb{Z}}$.
Suppose the contrary: $(i_a,i_a)\cdot a=b\in T\setminus
\mathscr{C}_{\mathbb{Z}}$. Assertion $(iii)$ implies that there
exist integers
\begin{equation*}
n_0(a)=\max\{n\mid(n,m)\in U(a)\cap \mathscr{C}_{\mathbb{Z}}\}
\qquad \hbox{and} \qquad m_0(a)=\max\{m\mid(n,m)\in U(a)\cap
\mathscr{C}_{\mathbb{Z}}\}.
\end{equation*}
Since $i_a\geqslant n_a$ we have that
\begin{equation*}
    (i_a,i_a)\cdot(n_0(a),m_0(a))=(i_a,i_a-n_0(a)+m_0(a)).
\end{equation*}
Let $W(b)$ be an open neighbourhood of $b$ in $T$ such that
$(i_a,i_a-n_0(a)+m_0(a))\notin W(b)$. Then the continuity of the
semigroup operation in $T$ implies that there exists an open
neighbourhood $V(a)$ of $a$ in $T$ such that
\begin{equation*}
    V(a)\subseteq U(a) \qquad \hbox{ and } \qquad \{(i_a,i_a)\}\cdot
    V(a)\subseteq W(b).
\end{equation*}
We fix an arbitrary element $(n,m)\in V(a)\cap
\mathscr{C}_{\mathbb{Z}}$. Then we have that
\begin{equation*}
    (i_a,i_a)\cdot(n,m)=(i_a,i_a-n+m)=(i_a,i_a-n_0(a)+m_0(a)),
\end{equation*}
a contradiction. The obtained contradiction implies that $a\in
R_{\mathscr{C}_{\mathbb{Z}}}$, and hence we have that
$L_{\mathscr{C}_{\mathbb{Z}}}\subseteq
R_{\mathscr{C}_{\mathbb{Z}}}$.

The proof of the inclusion $R_{\mathscr{C}_{\mathbb{Z}}}\subseteq
L_{\mathscr{C}_{\mathbb{Z}}}$ is similar.

Statement $(v)$ follows from statements $(i)-(iv)$ and
Proposition~\ref{proposition-1}$(v)$.

$(vi)$ Let $U(a)$ be an open neighbourhood of $a$ in $T$ such that
conditions $(a)$ and $(b)$ hold, and let $n_a$ be such integer as in
condition $(b)$. Then for all $(m_1,n_1),(m_2,n_2)\in U(a)\cap
\mathscr{C}_{\mathbb{Z}}$ we have that
\begin{equation*}
    (m_1,n_1)\cdot(n_a,n_a)=(m_1-n_1+n_a,n_a)=(m_2-n_2+n_a,n_a)=
    (m_2,n_2)\cdot(n_a,n_a),
\end{equation*}
and hence the continuity of the semigroup operation in $T$ implies
that $a\in L_{\mathscr{C}_{\mathbb{Z}}}$.

$(vii)$ Statements $(i)$ and $(iii)$ imply that $a\cdot(m,n)\in I$
and $(m,n)\cdot a\in I$ for all $a\in I$ and $(m,n)\in
\mathscr{C}_{\mathbb{Z}}$.

Fix arbitrary elements $a,b\in I$. We consider the following two
cases:
\begin{equation*}
    1)~~a\cdot b\in\mathscr{C}_{\mathbb{Z}} \qquad \hbox{ and }
    \qquad 2)~~a\cdot b\in L_{\mathscr{C}_{\mathbb{Z}}}.
\end{equation*}
In case 1) we put $a\cdot b=(m,n)\in\mathscr{C}_{\mathbb{Z}}$. Then
the continuity of the semigroup operation in $T$ implies that there
exist open neighbourhoods $U(a)$ and $U(b)$ of $a$ and $b$ in $T$,
respectively, such that
\begin{equation*}
    U(a)\cdot U(b)=\{(m,n)\}.
\end{equation*}
Since $a$ and $b$ are accumulation points of
$\mathscr{C}_{\mathbb{Z}}$ in $T$, we conclude that there exist
$(m_a,n_a)\in U(a)\cap\mathscr{C}_{\mathbb{Z}}$ and $(m_b,n_b)\in
U(b)\cap\mathscr{C}_{\mathbb{Z}}$. Hence we have that
\begin{equation*}
    (m_a,n_a)\cdot b\in \{(m_a,n_a)\}\cdot U(b)\subseteq U(a)\cdot
    U(b)=\{(m,n)\}
\end{equation*}
and
\begin{equation*}
    a\cdot(m_b,n_b)\in U(a)\cdot\{(m_b,n_b)\}\subseteq U(a)\cdot
    U(b)=\{(m,n)\}
\end{equation*}
This implies that $a,b\in L_{\mathscr{C}_{\mathbb{Z}}}$, a
contradiction.

Suppose case 2) holds and $a\cdot b=x\in
L_{\mathscr{C}_{\mathbb{Z}}}$. Then by statements $(i)$ and $(iii)$
we have that there exist an open neighbourhood $U(x)$ of $x$ in $T$
and an integer $n_x$ such that $m_1-n_1=m_2-n_2$, $m_1\leqslant n_x$
and $n_1\leqslant n_x$ for all $(m_1,n_1),(m_2,n_2)\in U(x)\cap
\mathscr{C}_{\mathbb{Z}}$. Also, the continuity of the semigroup
operation in $T$ implies that there exist open neighbourhoods $U(a)$
and $U(b)$ of $a$ and $b$ in $T$, respectively, such that
\begin{equation*}
    U(a)\cdot U(b)\subseteq U(x).
\end{equation*}
Since $U(a)\cap\mathscr{C}_{\mathbb{Z}}\neq\varnothing$ and
$U(b)\cap\mathscr{C}_{\mathbb{Z}}\neq\varnothing$, we can find
arbitrary elements $(m_a,n_a)\in U(a)\cap\mathscr{C}_{\mathbb{Z}}$
and $(m_b,n_b)\in U(b)\cap\mathscr{C}_{\mathbb{Z}}$. Then by
Proposition~\ref{proposition-1}$(vi)$ we have that
\begin{equation*}
    x_a-y_a+m_b-n_b=m_1-n_1 \qquad \hbox{ and } \qquad
    m_a-n_a+x_b-y_b=m_1-n_1
\end{equation*}
for all $(x_a,y_a)\in U(a)\cap\mathscr{C}_{\mathbb{Z}}$ and
$(x_b,y_b)\in U(b)\cap\mathscr{C}_{\mathbb{Z}}$. This implies that
there exist integers $k_a$ and $k_b$ such that
\begin{equation*}
    x_a-y_a=k_a \qquad \hbox{ and } \qquad
    x_b-y_b=k_b
\end{equation*}
for all $(x_a,y_a)\in U(a)\cap\mathscr{C}_{\mathbb{Z}}$ and
$(x_b,y_b)\in U(b)\cap\mathscr{C}_{\mathbb{Z}}$. Then by statement
$(vi)$ we have that $a,b\in L_{\mathscr{C}_{\mathbb{Z}}}$, a
contradiction.

The obtained contradictions imply that $a\cdot b\in I$, and hence we
get that the set $I$ is an ideal of $T$.

$(viii)$ Proposition~\ref{proposition-1}$(vi)$ and assertion $(vi)$
imply the following equalities:
\begin{equation*}
\{x\in T\mid x\cdot(b,b)=(a,b)\}=\{x\in T\mid (a,a)\cdot x=(a,b)\}=
\{x\in T\mid (a,a)\cdot x\cdot(b,b)=(a,b)\}.
\end{equation*}
Since by Corollary~\ref{corollary-4} every element $(a,b)$ of the
semigroup $\mathscr{C}_{\mathbb{Z}}$ is an isolated point in $T$,
the continuity of the semigroup operation in $T$ implies that
${\uparrow}(a,b)$ is a closed-and-open subset in $T$.

$(ix)$ Suppose that the set ${\uparrow}(a,b)\cap
L_{\mathscr{C}_{\mathbb{Z}}}$ is non-empty. Assuming that the set
${\uparrow}(a,b)\cap L_{\mathscr{C}_{\mathbb{Z}}}$ is non-singleton
implies that there exist distinct $x,y\in{\uparrow}(a,b)\cap
L_{\mathscr{C}_{\mathbb{Z}}}$. Then the Hausdorffness of $T$ implies
that there exist disjoint open neighbourhoods $U(x)$ and $U(y)$ of
$x$ and $y$ in $T$, respectively. By the continuity of the semigroup
operation in $T$ we can find open neighbourhoods $V(1_T)$, $V(x)$
and $V(y)$ of $1_T$, $x$ and $y$ in $T$, respectively, such that the
following conditions hold:
\begin{equation*}
    V(x)\subseteq U(x), \quad V(y)\subseteq U(y), \quad
    V(1_T)\cdot V(x)\subseteq U(x) \quad \hbox{and} \quad
    V(1_T)\cdot V(y)\subseteq U(y).
\end{equation*}
By assertions $(i)-(iii)$ we can find the integers $n,n_1,n_2,m_1$
and $m_2$ such that
\begin{equation*}
\begin{split}
  (n,n)\in V(1_T),  & \quad  (n_1,n_2)\in V(x), \quad (m_1,m_2)\in
    V(y), \quad n_1-n_2=m_1-m_2, \\
    & \quad n\geqslant n_1 \quad
    \hbox{and} \quad n\geqslant m_1.
\end{split}
\end{equation*}
Then we have that
\begin{equation*}
    (n,n)\cdot(n_1,n_2)=(n,n-n_1+n_2)=(n,n-m_1+m_2)=
    (n,n)\cdot(m_1,m_2),
\end{equation*}
and hence $(V(1_T)\cdot V(x))\cdot(V(1_T)\cdot
V(y))\neq\varnothing$, a contradiction. The obtained contradiction
implies that $x=y$.

$(x)$ Statement $(vii)$ implies that
$T\setminus\left(I\cup\mathscr{C}_{\mathbb{Z}}\right)=
L_{\mathscr{C}_{\mathbb{Z}}}$. Let $\mathbb{Z}$ be the additive
group of integers. We define a map $\mathfrak{h}\colon
L_{\mathscr{C}_{\mathbb{Z}}}\rightarrow\mathbb{Z}$ as follows:
\begin{equation*}
\begin{split}
  (x)\mathfrak{h}=n \qquad &
     \hbox{if and only if there exists a neighbourhood }
     U(x) \hbox{ of } x \hbox{ in } T \hbox{ such that } \\
    & a-b=n, \hbox{ for all } (a,b)\in U(x)\cap\mathscr{C}_{\mathbb{Z}},
\end{split}
\end{equation*}
where $x\in L_{\mathscr{C}_{\mathbb{Z}}}$. We observe that
assertions $(i)-(v)$ imply that the map $\mathfrak{h}$ is well
defined. Also, Proposition~\ref{proposition-1} implies that
$\mathfrak{h}\colon L_{\mathscr{C}_{\mathbb{Z}}}
\rightarrow\mathbb{Z}$ is a monomorphism, and hence
$L_{\mathscr{C}_{\mathbb{Z}}}$ is a submonoid of $\mathbb{Z}$. In
the case when a maximal subgroup of $L_{\mathscr{C}_{\mathbb{Z}}}$
is non-singleton Proposition~\ref{proposition-9} implies that
$\left(L_{\mathscr{C}_{\mathbb{Z}}}\right)\mathfrak{h}$ is a cyclic
subgroup of $\mathbb{Z}$. This completes the proof of our assertion.

$(xi)$ Assertion $(v)$ implies that
\begin{equation*}
    {\uparrow}\mathscr{C}_{\mathbb{Z}}=\{x\in T\mid \hbox{there
    exists } y\in\mathscr{C}_{\mathbb{Z}} \hbox{ such that } x\cdot
    y\in\mathscr{C}_{\mathbb{Z}}\}=
    \bigcup_{(a,b)\in\mathscr{C}_{\mathbb{Z}}}{\uparrow}(a,b).
\end{equation*}
Then assertion $(viii)$ implies that
${\uparrow}\mathscr{C}_{\mathbb{Z}}$ is an open subset in $T$ and
hence by assertion $(vii)$ we get that the ideal $I$ is a closed
subset of $T$.

Assertion $(xii)$ follows from $(x)$.
\end{proof}


\section{On a closure of the
semigroup $\mathscr{C}_{\mathbb{Z}}$ in a locally compact
topological inverse semigroup}


For every non-negative integer $k$ by $k\mathbb{Z}$ we denote a
subgroup of the additive group of integers $\mathbb{Z}$ which is
generated by an element $k\in\mathbb{Z}$. We observe if $k=0$ then
the group $k\mathbb{Z}$ is trivial. Also, we denote $G_0=\mathbb{Z}$
and $G_1(k)=k\mathbb{Z}$ for a positive integer $k$.

The following five examples illustrate distinct structures of a
closure of the semigroup $\mathscr{C}_{\mathbb{Z}}$ in a locally
compact topological inverse semigroup.

\begin{example}\label{example-11}
Let be $S_1=G_1(0)\sqcup\mathscr{C}_{\mathbb{Z}}$. Then $G_1(0)$ is
a trivial group and we put $\{e_1\}=G_1(0)$. We extend the semigroup
operation from $\mathscr{C}_{\mathbb{Z}}$ onto $S_1$ as follows:
\begin{equation*}
     e_1\cdot(a,b)=(a,b)\cdot e_1=(a,b)\in\mathscr{C}_{\mathbb{Z}}
     \qquad \hbox{ and } \qquad e_1\cdot e_1=e_1,
\end{equation*}
i.e., $S_1$ is the semigroup $\mathscr{C}_{\mathbb{Z}}$ with the
adjoined unit $e_1$. We fix an arbitrary decreasing sequence
$\{m_i\}_{i\in\mathbb{N}}$ of negative integers and for every
positive integer $n$ we put
\begin{equation*}
    U_n(e_1)=\{e_1\}\cup\left\{(m_i,m_i)\in\mathscr{C}_{\mathbb{Z}}\mid
    i\geqslant n\right\}.
\end{equation*}
Then we determine a topology $\tau_1$ on $S_1$ as follows:
\begin{itemize}
  \item[1)] all elements of the semigroup $\mathscr{C}_{\mathbb{Z}}$
   are isolated points in $(S_1,\tau_1)$; \; and
  \item[2)] the family $\mathscr{B}_1(e_1)=\left\{U_n(e_1)\mid
   n\in\mathbb{N}\right\}$ is a base of the topology $\tau_1$ at the
   point $e_1\in G_1(0)\subseteq S_1$.
\end{itemize}

Then for every positive integer $n$ we have that
\begin{equation*}
    U_n(e_1)\cdot U_n(e_1)=U_n(e_1) \qquad \hbox{ and }
    \qquad \left(U_n(e_1)\right)^{-1}=U_n(e_1).
\end{equation*}
Let $(m,n)$ be an arbitrary element of the semigroup
$\mathscr{C}_{\mathbb{Z}}$. We fix a positive integer $i_{(m,n)}$
such that $m_{i_{(m,n)}}\leqslant m$ and $m_{i_{(m,n)}}\leqslant n$.
Then we have that
\begin{equation*}
    U_{i_{(m,n)}}(e_1)\cdot\{(m,n)\}=\{(m,n)\} \qquad \hbox{ and }
    \qquad \{(m,n)\}\cdot U_{i_{(m,n)}}(e_1)=\{(m,n)\}.
\end{equation*}
Hence we get that $(S_1,\tau_1)$ is a topological inverse semigroup.
Obviously, $(S_1,\tau_1)$ is a Hausdorff locally compact space.
\end{example}

\begin{example}\label{example-12}
Let $k$ and $n$ be any positive integers such that
$n\in\{1,\ldots,k\}$ is a divisor of $k$ and we put $k=n\cdot s$,
where $s$ is some positive integer. We put $S_2=G_1(k)\sqcup
\mathscr{C}_{\mathbb{Z}}$. Later an element of the group
$G_1(k)=k\mathbb{Z}$ will be denote by $ki$, where $i\in\mathbb{Z}$.
We extend the semigroup operation from $\mathscr{C}_{\mathbb{Z}}$
onto $S_2$ by the following way:
\begin{equation*}
    ki\cdot(a,b)=(-ki+a,b)\in\mathscr{C}_{\mathbb{Z}} \qquad
    \hbox{ and } \qquad (a,b)\cdot
    ki=(a,b+ki)\in\mathscr{C}_{\mathbb{Z}},
\end{equation*}
for arbitrary $(a,b)\in\mathscr{C}_{\mathbb{Z}}$ and $ki\in G_1(k)$.
To see that the extended binary operation is associative we need
only check six possibilities, the other being evident.

Then for arbitrary $ki_1, ki_2\in G_1(k)$ and $(a,b),
(c,d)\in\mathscr{C}_{\mathbb{Z}}$ we have that:
\begin{itemize}
  \item[$1)$] $(ki_1\cdot ki_2)\cdot(a,b)=(ki_1+ki_2)(a,b)=
   (-ki_1-ki_2+a,b)=
   ki_1\cdot(-ki_2+a,b)\quad$ $=ki_1\cdot\left(ki_2\cdot(a,b)\right)$;

  \item[$2)$] $(a,b)\cdot(ki_1\cdot ki_2)=(a,b)\cdot(ki_1+ki_2)=
   (a,b+ki_1+ki_2)=(a,b+ki_1)\cdot ki_2=\left((a,b)
   \cdot ki_1\right)\cdot ki_2$;

  \item[$3)$] $\left(ki_1\cdot(a,b)\right)\cdot ki_2=
   (-ki_1+a,b)\cdot ki_2=(-ki_1+a,b+ki_2)=ki_1\cdot(a,b+ki_2)\quad$
   $=ki_1\cdot\left((a,b)\cdot ki_2\right)$;

  \item[$4)$] $\left(ki_1\cdot(a,b)\right)\cdot(c,d)=
   (-ki_1+a,b)\cdot(c,d)=
   \left\{
     \begin{array}{ll}
       (-ki_1+a-b+c,d), & \hbox{if~~} b\leqslant c;\\
       (-ki_1+a,b-c+d), & \hbox{if~~} b\geqslant c
     \end{array}
   \right.\quad
   $ \newline
   $=
   \left\{
     \begin{array}{ll}
       ki_1\cdot(a-b+c,d), & \hbox{if~~} b\leqslant c;\\
       ki_1\cdot(a,b-c+d), & \hbox{if~~} b\geqslant c
     \end{array}
   \right.
   =ki_1\cdot\left((a,b)\cdot(c,d)\right)$;

  \item[$5)$] $\left((a,b)\cdot(c,d)\right)\cdot ki_1=
   \left\{
     \begin{array}{ll}
       (a-b+c,d)\cdot ki_1, & \hbox{if~~} b\leqslant c;\\
       (a,b-c+d)\cdot ki_1, & \hbox{if~~} b\geqslant c
     \end{array}
   \right.
   \quad $ \newline
   $=
   \left\{
     \begin{array}{ll}
       (a-b+c,d+ki_1), & \hbox{if~~} b\leqslant c;\\
       (a,b-c+d+ki_1), & \hbox{if~~} b\geqslant c
     \end{array}
   \right.
   =(a,b)\cdot\left(c,d+ki_1\right)=(a,b)\cdot\left((c,d)\cdot
   ki_1\right)$;

  \item[$6)$] $\left((a,b)\cdot ki_1\right)\cdot(c,d)=
   \left(a,b+ki_1\right)\cdot(c,d)=
   \left\{
     \begin{array}{ll}
       (a-b-ki_1+c,d), & \hbox{if~~} b+ki_1\leqslant c;\\
       (a,b+ik_1-c+d), & \hbox{if~~} b+ki_1\geqslant c
     \end{array}
   \right.
   $ \newline
   $=
   \left\{
     \begin{array}{ll}
       (a-b-ki_1+c,d), & \hbox{if~~} b\leqslant-ki_1+c;\\
       (a,b+ki_1-c+d), & \hbox{if~~} b\geqslant-ki_1+c
     \end{array}
   \right.
   =\left(a,b\right)\cdot(-ki_1+c,d)=
   (a,b)\cdot\left(ki_1\cdot(c,d)\right)$.
\end{itemize}
Also simple verifications show that $S_2$ is an inverse semigroup.

Let $ki$ be an arbitrary element of the group $G_1(k)$. For every
positive integer $j$ we denote
\begin{equation*}
    U_j^n(ki)=\{ki\}\cup\left\{(-nq,-nq+ki)\mid q\geqslant j,
    q\in\mathbb{N}\right\}.
\end{equation*}

We determine a topology $\tau_2$ on $S_2$ as follows:
\begin{itemize}
  \item[1)] all elements of the semigroup $\mathscr{C}_{\mathbb{Z}}$
   are isolated points in $(S_2,\tau_2)$; \; and
  \item[2)] the family $\mathscr{B}_2(ki)=\left\{U_j^n(ki)\mid
   j\in\mathbb{N}\right\}$ is a base of the topology $\tau_2$ at the
   point $ki\in G_1(k)\subseteq S_2$.
\end{itemize}

Then for every positive integer $j$ we have that
\begin{equation*}
    U_j^n(ki_1)\cdot U_{j-i_1s}^n(ki_2)\subseteq U_j^n(ki_1+ki_2)
    \qquad \hbox{ and } \qquad
    \left(U_j^n(ki_1)\right)^{-1}=U_j^n(-ki_1),
\end{equation*}
for $ki_1,ki_2\in G_1(k)$.

Let $(a,b)$ be an arbitrary element of the semigroup
$\mathscr{C}_{\mathbb{Z}}$ and $ki\in G_1(k)$. Then we have that
\begin{equation*}
    U_j^n(ki)\cdot\{(a,b)\}=\{(a-ki,b)\}
    \qquad \hbox{ and } \qquad
    \{(a,b)\}\cdot U_j^n(ki)=\{(a,b+ki)\},
\end{equation*}
for every positive integer $j$ such that
$nj\geqslant\max\{-b;ki-a\}$.

Therefore $(S_2,\tau_2)$ is a topological inverse semigroup, and
moreover the topological space $(S_2,\tau_2)$ is Hausdorff and
locally compact.
\end{example}

\begin{example}\label{example-13}
We put $S_3=\mathscr{C}_{\mathbb{Z}}\sqcup G_0$ and extend the
semigroup operation from the semigroup $\mathscr{C}_{\mathbb{Z}}$
onto $S_3$ by the following way:
\begin{equation*}
    (a,b)\cdot n=n\cdot(a,b)=n+b-a\in G_0,
\end{equation*}
for all $(a,b)\in\mathscr{C}_{\mathbb{Z}}$ and $n\in G_0$. To see
that the extended binary operation is associative we need only check
two possibilities, the other being evident.

Then for arbitrary $m,n\in G_0$ and $(a,b),
(c,d)\in\mathscr{C}_{\mathbb{Z}}$ we have that:
\begin{itemize}
  \item[$1)$] $\left(n\cdot(a,b)\right)\cdot(c,d)=(n+b-a)\cdot(c,d)=
   n+b-a+d-c=
   \left\{
     \begin{array}{ll}
       n\cdot(a-b+c,d), & \hbox{if~} b\leqslant c;\\
       n\cdot(a,b-c+d), & \hbox{if~} b\geqslant c
     \end{array}
   \right.
   $ \newline
   $=n\cdot\left((a,b)\cdot(c,d)\right)$;

  \item[$2)$] $(m\cdot n)\cdot(a,b)=m+n+b-a=m\cdot(n+b-a)=
   m\cdot\left(n\cdot(a,b)\right)$.
\end{itemize}
This completes the proof of the associativity of such defined binary
operation on $S_3$. Also, we observe that $S_3$ with such defined
semigroup operation is an inverse semigroup.

For every positive integer $n$ and every element $k\in G_0$ we put:
\begin{equation*}
 U_n(k)=
\left\{
  \begin{array}{ll}
    \{k\}\cup\left\{(a,a+k)\mid
   a=n,n+1,n+2,\ldots\right\}, & \hbox{if~} k\geqslant 0;\\
    \{k\}\cup\left\{(a-k,a)\mid
   a=n,n+1,n+2,\ldots\right\}, & \hbox{if~} k\leqslant 0.
  \end{array}
\right.
\end{equation*}

We determine a topology $\tau_3$ on $S_3$ as follows:
\begin{itemize}
  \item[1)] all elements of the semigroup $\mathscr{C}_{\mathbb{Z}}$
   are isolated points in $(S_3,\tau_3)$; \; and
  \item[2)] the family $\mathscr{B}_3(k)=\left\{U_n(k)\mid
   n\in\mathbb{N}\right\}$ is a base of the topology $\tau_3$ at the
   point $k\in G_0\subseteq S_3$.
\end{itemize}

Then for all $k_1,k_2\in G_0$ we have that
\begin{equation*}
    U_{2n}(k_1)\cdot U_{2n}(k_2)\subseteq U_{n}(k_1+k_2),
\end{equation*}
for every positive integer
$n\geqslant\max\left\{\left|k_1\right|,\left|k_2\right|\right\}$,
and
\begin{equation*}
\left(U_i(k_1)\right)^{-1}=U_i(-k_1),
\end{equation*}
for every positive integer $i$. Also, for arbitrary $(a,b)\in
\mathscr{C}_{\mathbb{Z}}$ and $k\in G_0$ we have that
\begin{equation*}
    (a,b)\cdot U_{2n}(k)\subseteq U_n(k+b-a) \qquad \hbox{ and }
    \qquad U_{2n}(k)\cdot (a,b)\subseteq U_n(k+b-a),
\end{equation*}
for every positive integer $n\geqslant\max\left\{\left|a\right|,
\left|b\right|, \left|k\right|\right\}$.

This completes the proof that $(S_3,\tau_3)$ is a topological
inverse semigroup. Obviously, $(S_3,\tau_3)$ is a Hausdorff locally
compact space.
\end{example}

\begin{example}\label{example-14}
Let be $S_4=G_1(0)\sqcup S_3$, where the group $G_1(0)$ and the
semigroup $S_3$ are defined in Example~\ref{example-11} and
Example~\ref{example-13}, respectively. We extend the semigroup
operation from $S_3$ onto $S_4$ as follows:
\begin{equation*}
    e_1\cdot x=x\cdot e_1=x\in\mathscr{C}_{\mathbb{Z}} \qquad
    \hbox{ and } \qquad e_1\cdot e_1=e_1,
\end{equation*}
i.e., $S_4$ is the semigroup $S_3$ with the adjoined unit $e_1$.

Let $\tau_4$ be a topology on $S_4$ which is generated by the family
$\tau_1\cup\tau_3$ (see Examples~\ref{example-11} and
\ref{example-13}). Then for every element $k_0\in G_0$ and every
positive integers $n_1$ and $n_0$ we have that the following
inclusions hold:
\begin{equation*}
    U_{n_1}(e_1)\cdot U_{n_0}(k_0)\subseteq U_{n_0}(k_0) \qquad
    \hbox{ and } \qquad U_{n_0}(k_0)\cdot U_{n_1}(e_1)\subseteq
    U_{n_0}(k_0),
\end{equation*}
where $U_{n_1}(e_1)\in\mathscr{B}_1(e_1)$ and $U_{n_0}(k_0)\in
\mathscr{B}_3(k_0)$ (see Examples~\ref{example-11} and
\ref{example-13}). These inclusions and Examples~\ref{example-11}
and \ref{example-13} imply that $(S_4,\tau_4)$ is a Hausdorff
topological inverse semigroup. Obviously, $(S_4,\tau_4)$ is a
locally compact space.
\end{example}

\begin{example}\label{example-15}
Let $k$ and $n$ be such positive integers as in
Example~\ref{example-12}. We put $S_5=G_1(k)
\sqcup\mathscr{C}_{\mathbb{Z}}\sqcup G_0$ and extend semigroup
operation from $S_2$ and $S_3$ onto $S_5$ as follows. Later we
denote elements of groups $G_1(K)$ and $G_0$ by $(ki)^1$ and
$(n)^0$, respectively. We put
\begin{equation*}
    (ki)^1\cdot(n)^0=(n)^0\cdot(ki)^1=(ki+n)^0\in G_0,
\end{equation*}
for all $(ki)^1\in G_1(k)$ and $(n)^0\in G_0$. To see that the
extended binary operation is associative we need only check twelve
possibilities, the other either are evident or are proved in
Examples~\ref{example-12} and ~\ref{example-13}.

Then for arbitrary $(ki_1)^1, (ki_2)^1\in G_1(k)$,
$(n_1)^0,(n_2)^0\in G_0$ and $(a,b)\in \mathscr{C}_{\mathbb{Z}}$ we
have that:
\begin{itemize}
  \item[1)] $\left((n_1)^0\cdot(n_2)^0\right)\cdot(ki_1)^1=
   \left(n_1+n_2\right)^0\cdot(ki_1)^1=\left(n_1+n_2+ki_1\right)^0=
   (n_1)^0\cdot\left(n_2+ki_1\right)^0$ \newline
   $=(n_1)^0\cdot\left((n_2)^0\cdot(ki_1)^1\right)$;

  \item[2)] $\left((n_1)^0\cdot(ki_1)^1\right)\cdot(n_2)^0=
   \left(n_1+ki_1\right)^0\cdot(n_2)^0=\left(n_1+ki_1+n_2\right)^0=
   (n_1)^0\cdot\left(ki_1+n_2\right)^0$ \newline
   $=(n_1)^0\cdot\left((ki_1)^1\cdot(n_2)^0\right)$;

  \item[3)] $\left((n_1)^0\cdot(ki_1)^1\right)\cdot(ki_2)^1=
   \left(n_1+ki_1\right)^0\cdot(ki_2)^1=\left(n_1+ki_1+ki_2\right)^0=
   (n_1)^0\cdot\left(ki_1+ki_2\right)^1$ \newline
   $=(n_1)^0\cdot\left((ki_1)^1\cdot(ki_2)^1\right)$;

  \item[4)] $\left((n_1)^0\cdot(ki_1)^1\right)\cdot(a,b)=
   \left(n_1+ki_1\right)^0\cdot(a,b)=\left(n_1+ki_1+b-a\right)^0=
   (n_1)^0\cdot\left(-ki_1+a,b\right)$ \newline
   $=(n_1)^0\cdot\left((ki_1)^1\cdot(a,b)\right)$;

  \item[5)] $\left((n_1)^0\cdot(a,b)\right)\cdot(ki_1)^1=
   \left(n_1+b-a\right)^0\cdot(ki_1)^1=\left(n_1+b-a+ki_1\right)^0=
   (n_1)^0\cdot\left(a,b+ki_1\right)$ \newline
   $=(n_1)^0\cdot\left((a,b)\cdot(ki_1)^1\right)$;

  \item[6)] $\left((ki_1)^1\cdot(n_1)^0\right)\cdot(n_2)^0=
   \left(ki_1+n_1\right)^0\cdot(n_2)^0=\left(ki_1+n_1+n_2\right)^0=
   (ki_1)^1\cdot\left(n_1+n_2\right)^0$ \newline
   $=(ki_1)^1\cdot\left((n_1)^0\cdot(n_2)^0\right)$;

  \item[7)] $\left((ki_1)^1\cdot(n_1)^0\right)\cdot(ki_2)^1=
   \left(ki_1+n_1\right)^0\cdot(ki_2)^1=\left(ki_1+n_1+ki_2\right)^0=
   (ki_1)^1\cdot\left(n_1+ki_2\right)^0$ \newline
   $=(ki_1)^1\cdot\left((n_1)^0\cdot(ki_2)^1\right)$;

  \item[8)] $\left((ki_1)^1\cdot(n_1)^0\right)\cdot(a,b)=
   \left(ki_1+n_1\right)^0\cdot(a,b)=\left(ki_1+n_1+b-a\right)^0=
   (ki_1)^1\cdot\left(n_1+b-a\right)^0$ \newline
   $=(ki_1)^1\cdot\left((n_1)^0\cdot(a,b)\right)$;

  \item[9)] $\left((ki_1)^1\cdot(ki_2)^1\right)\cdot(n_1)^0=
   \left(ki_1+ki_2\right)^1\cdot(n_1)^0=\left(ki_1+ki_2+n_1\right)^0=
   (ki_1)^1\cdot\left(ki_2+n_1\right)^0$ \newline
   $=(ki_1)^1\cdot\left((ki_2)^1\cdot(n_1)^0\right)$;

  \item[10)] $\left((ki_1)^1\cdot(a,b)\right)\cdot(n_1)^0=
   (-ki_1+a,b)\cdot(n_1)^0=\left(ki_1+b-a+n_1\right)^0$
   \newline
   $=(ki_1)^1\cdot\left(b-a+n_1\right)^0=
   (ki_1)^1\cdot\left((a,b)\cdot(n_1)^0\right)$;

  \item[11)] $\left((a,b)\cdot(n_1)^0\right)\cdot(ki_1)^1=
   (b-a+n_1)^0\cdot(ki_1)^1=\left(b-a+n_1+ki_1\right)^0=
   (a,b)\cdot\left(n_1+ki_1\right)^0$ \newline
   $=(a,b)\cdot\left((n_1)^0\cdot(ki_1)^1\right)$;

  \item[12)] $\left((a,b)\cdot(ki_1)^1\right)\cdot(n_1)^0=
   (a,b+ki_1)^0\cdot(n_1)^0=\left(b+ki_1-a+n_1\right)^0=
   (a,b)\cdot\left(ki_1+n_1\right)^0$ \newline
   $=(a,b)\cdot\left((ki_1)^1\cdot(n_1)^0\right)$.
\end{itemize}

This completes the proof of the associativity of such defined binary
operation on $S_5$. Also, we observe that $S_5$ with such defined
semigroup operation is an inverse semigroup.

Let $\tau_5$ be a topology on $S_5$ which is generated by the family
$\tau_2\cup\tau_3$ (see Examples~\ref{example-12} and
\ref{example-13}). Also Examples~\ref{example-12} and
\ref{example-13} imply that it is sufficient to show that the
semigroup operation in $S_5$ is continuous in cases $(ki)^1\cdot
(n)^0$ and $(n)^0\cdot(ki)^1$, where $(n)^0\in G_0$ and $(ki)^1\in
G_1(k)$. Then for every positive integer
$p\geqslant\max\left\{|ki|,|n|\right\}$ we have that
\begin{equation*}
 U_{2p}\!\left((ki)^1\right)\cdot U_{2p}\!\left((n)^0\right)\subseteq
 U_{p}\!\left((ki+n)^0\right) \qquad \hbox{ and } \qquad
 U_{2p}\!\left((n)^0\right)\cdot U_{2p}\!\left((ki)^1\right)\subseteq
 U_{p}\!\left((ki+n)^0\right).
\end{equation*}
This completes the proof that $(S_5,\tau_5)$ is a topological
inverse semigroup. Obviously, $(S_5,\tau_5)$ is a locally compact
space.
\end{example}



\begin{theorem}\label{theorem-16-0}
Let $T$ be a Hausdorff topological inverse semigroup. If $T$
contains $\mathscr{C}_{\mathbb{Z}}$ as a dense subsemigroup and
$I=T\setminus{\uparrow}\mathscr{C}_{\mathbb{Z}}\neq \varnothing$,
then the following assertions hold:
\begin{itemize}
  \item[$(i)$] $E(T)$ is a countable linearly ordered semilattice;

  \item[$(ii)$] $E(T)\cap\left(T\setminus
   {\uparrow}\mathscr{C}_{\mathbb{Z}}\right)$ is a singleton set;

  \item[$(iii)$] $T\setminus{\uparrow}\mathscr{C}_{\mathbb{Z}}$ is
   a subgroup in $T$.
\end{itemize}
\end{theorem}

\begin{proof}
$(i)$ By Proposition~II.3 from \cite{EberhartSelden1969} we have
that $\operatorname{cl}_{T}(E(\mathscr{C}_{\mathbb{Z}}))=E(T)$ and
since the closure of a linearly ordered subsemilattice in a
topological semilattice is a linearly ordered subsemilattice too
(see \cite[Lemma~1]{GutikRepovs2008}) we get that $E(T)$ is a
linearly ordered semilattice. Then the semilattice operation in
$E(T)$ implies that the sets
$E(T)\setminus\!\displaystyle\bigcup_{e\in
E(\mathscr{C}_{\mathbb{Z}})}\!{\downarrow}e$ and
$E(T)\setminus\!\displaystyle\bigcup_{e\in
E(\mathscr{C}_{\mathbb{Z}})}\!{\uparrow}e$ are either singleton or
empty. This completes the proof of our assertion.

Assertion $(ii)$ follows from assertion $(i)$.

$(iii)$ Since $T$ is an inverse semigroup and $\overline{e}$ is a
minimal idempotent in $E(T)$ we conclude that the
$\mathscr{H}$-class $H_{\overline{e}}$ which contains $\overline{e}$
coincides with the ideal $I= T\setminus
{\uparrow}\mathscr{C}_{\mathbb{Z}}$. Indeed, if there exist $x\in I$
and an $\mathscr{H}$-class $H_x\subseteq I$ in $T$ such that $x\in
H_x\neq H_{\overline{e}}$, then since $T$ is an inverse semigroup we
have that there exists an idempotent $e\in T$ such that either
$xx^{-1}=e\in {\uparrow}\mathscr{C}_{\mathbb{Z}}$ or $x^{-1}x=e\in
{\uparrow}\mathscr{C}_{\mathbb{Z}}$. If $xx^{-1}=e\in
{\uparrow}\mathscr{C}_{\mathbb{Z}}$, then we have that
$x=xx^{-1}x=ex\in eT$, and since $T$ is an inverse semigroup
Theorem~1.17 from~\cite{CP} implies $e\in xT$, a contradiction.
Similar arguments show that $x^{-1}x\neq e\in
{\uparrow}\mathscr{C}_{\mathbb{Z}}$. Hence assertion $(ii)$ implies
that $xx^{-1}=x^{-1}x=\overline{e}$ and hence $x\in H_x=
H_{\overline{e}}$.
\end{proof}

The following theorem describes the structure of a closure of the
semigroup $\mathscr{C}_{\mathbb{Z}}$ in a locally compact
topological inverse semigroup $T$, i.e., it gives the description of
the non-empty ideal $I=T\setminus{\uparrow}\mathscr{C}_{\mathbb{Z}}$
in the remainder of $\mathscr{C}_{\mathbb{Z}}$ in $T$.

\begin{theorem}\label{theorem-16}
Let $T$ be a Hausdorff locally compact topological inverse
semigroup. If $T$ contains $\mathscr{C}_{\mathbb{Z}}$ as a dense
subsemigroup and $I=T\setminus{\uparrow}\mathscr{C}_{\mathbb{Z}}\neq
\varnothing$, then the following assertions hold:
\begin{itemize}
  \item[$(i)$] ${\downarrow}e_n$ is a compact subsemilattice in
   $E(T)$ for every idempotent $e_n=(n,n)\in
   \mathscr{C}_{\mathbb{Z}}$, $n\in\mathbb{Z}$;

  \item[$(ii)$] $T\setminus{\uparrow}\mathscr{C}_{\mathbb{Z}}$ is
   isomorphic to the discrete additive group of integers;

  \item[$(iii)$] if $\overline{e}$ is a unit of
   $T\setminus{\uparrow}\mathscr{C}_{\mathbb{Z}}$, then the map
   $\mathfrak{h}\colon\mathscr{C}_{\mathbb{Z}}\rightarrow
   T\setminus{\uparrow}\mathscr{C}_{\mathbb{Z}}$ which is defined by
   the formula $((a,b))\mathfrak{h}=(a,b)\cdot\overline{e}$ is the
   natural homomorphism generated by the minimal group congruence
   $\mathfrak{C}_{mg}$ on the semigroup $\mathscr{C}_{\mathbb{Z}}$;

  \item[$(iv)$] the subsemigroup $S=\mathscr{C}_{\mathbb{Z}}\cup I$
   is topologically isomorphic to the topological inverse semigroup
   $(S_3,\tau_3)$ from Example~\ref{example-13}.
\end{itemize}
\end{theorem}

\begin{proof}
$(i)$ We show that ${\downarrow}e_0$ is a compact subset in $E(T)$
for $e_0=(0,0)$. By assertion $(ii)$ of Theorem~\ref{theorem-16-0}
we get that the set
$E(T)\cap\left(T\setminus{\uparrow}\mathscr{C}_{\mathbb{Z}}\right)$
is singleton and we put $\left\{\overline{e}\right\}= E(T)\cap
\left(T\setminus {\uparrow}\mathscr{C}_{\mathbb{Z}}\right)$. Then
$\overline{e}$ is a smallest idempotent in $E(T)$. By Theorem~1.5
from \cite[Vol.~1]{CHK} we have that $E(T)$ is a closed subset in
$T$, and hence by Theorem~3.3.9 from \cite{Engelking1989} we get
that $E(T)$ is a locally compact space. Suppose the contrary:
${\downarrow}e_0$ is not a compact subset in $E(T)$. Since
Corollary~\ref{corollary-4} implies that every element of the
semigroup $\mathscr{C}_{\mathbb{Z}}$ is an isolated point in $T$ and
hence so it is in $E(T)$, we get that there exists an open
neighbourhood $U(\overline{e})$ of $\overline{e}$ in $E(T)$ such
that the set ${\downarrow}e_0\setminus U(\overline{e})$ is an
infinite discrete subspace of $E(T)$, $U(\overline{e})\subseteq
E(T)\setminus{\uparrow}e_0$ and
$\operatorname{cl}_{E(T)}(U(\overline{e}))=U(\overline{e})$ is a
compact subset of $E(T)$. Then for every positive integer $i$ there
exists an integer $j\geqslant i$ such that $(j,j)\notin
U(\overline{e})$ and $(j+1,j+1)\in U(\overline{e})$. Then the
semigroup operation in $\mathscr{C}_{\mathbb{Z}}$ implies that by
induction we can construct an infinite subset $M\subseteq
{\downarrow}e_0\setminus\left\{\overline{e}\right\}$ of $E(T)$ such
that $M\subseteq U(\overline{e})\setminus \left\{ \overline{e}
\right\}$ and $\{(0,1)\}\cdot M\cdot\{(1,0)\}\subseteq
{\downarrow}e_0 \setminus U(\overline{e})$. Since the set
$U(\overline{e})$ is compact and the set $M\subseteq
U(\overline{e})\setminus \left\{ \overline{e}\right\}$ contains only
isolated points from $E(\mathscr{C}_{\mathbb{Z}})$, we conclude that
$\overline{e}\in \operatorname{cl}_{T}(M)$. Since
${\downarrow}e_0\setminus U(\overline{e})$ is a closed subset of
$E(T)$ we have that the continuity of the semigroup operation in $T$
and Proposition~1.4.1 from \cite{Engelking1989} imply that
\begin{equation*}
 \overline{e}\in \{(0,1)\}\cdot\operatorname{cl}_{T}(M)
 \cdot\{(1,0)\}\subseteq \operatorname{cl}_{T}\left(\{(0,1)\}\cdot M
 \cdot\{(1,0)\}\right)\subseteq{\downarrow}e_0\setminus
 U(\overline{e}),
\end{equation*}
which contradicts $\overline{e}\in U(\overline{e})$. The obtained
contradiction implies that the set ${\downarrow}e_0\setminus
U(\overline{e})$ is finite, and hence the set ${\downarrow}e_0$ is
compact. Since for every integer $n$ the set
${\downarrow}e_n\setminus{\downarrow}e_0$ is either finite or empty
and $e_n$ is an isolated point in $E(T)$ we conclude that
${\downarrow}e_n$ is a compact subsemilattice of $E(T)$.

$(ii)$ By assertion $(i)$ we have that $\overline{e}$ is an
accumulation point of the subsemigroup $\mathscr{C}_{\mathbb{N}}[0]$
in $T$. Since by Theorem~3.3.9 from \cite{Engelking1989} a closed
subset of a locally compact space is a locally compact subspace too,
and by Proposition~\ref{proposition-1}$(viii)$ the semigroup
$\mathscr{C}_{\mathbb{N}}[0]$ is  isomorphic to the bicyclic
semigroup, Proposition~V.3 from \cite{EberhartSelden1969} implies
that the subset
$\operatorname{cl}_{T}\left(\mathscr{C}_{\mathbb{N}}[0]\right)\setminus
\mathscr{C}_{\mathbb{N}}[0]$ is a non-singleton subgroup of $T$. By
Corollary~\ref{corollary-4} we get that $\mathscr{C}_{\mathbb{Z}}$
is an open discrete subsemigroup of $T$ and hence we get that
$\operatorname{cl}_{T}\left(\mathscr{C}_{\mathbb{N}}[0]\right)\setminus
\mathscr{C}_{\mathbb{N}}[0]\subseteq
\operatorname{cl}_{T}\left(\mathscr{C}_{\mathbb{Z}}\right)\setminus
\mathscr{C}_{\mathbb{Z}}$.

By assertion $(iii)$ of Theorem~\ref{theorem-16-0} we have that
$I=T\setminus {\uparrow}\mathscr{C}_{\mathbb{Z}}$ is a non-singleton
subgroup in $T$. Since $T$ is a topological inverse semigroup we get
that $I$ is a topological group. Then by
Proposition~\ref{proposition-10}$(xi)$ we have that $I$ is a closed
subset of $T$ and hence by Theorem~3.3.9 from \cite{Engelking1989}
we get that $I$ is a locally compact topological group.

Later we show that $(a,b)\cdot\overline{e}=\overline{e}\cdot(a,b)$
for every $(a,b)\in\mathscr{C}_{\mathbb{Z}}$. Suppose the contrary:
there exists $(a,b)\in\mathscr{C}_{\mathbb{Z}}$ such that
$(a,b)\cdot\overline{e}\neq\overline{e}\cdot(a,b)$. Without loss of
generality we can assume that $a\leqslant b$ in $\mathbb{Z}$. Then
the Hausdorffness of the space $T$ implies that there exist open
neighbourhoods $U((a,b)\cdot\overline{e})$ and
$U(\overline{e}\cdot(a,b))$ of the points $(a,b)\cdot\overline{e}$
and $\overline{e}\cdot(a,b)$ in $T$ such that
$U((a,b)\cdot\overline{e})\cap U(\overline{e}\cdot(a,b))=
\varnothing$. Then the continuity of the semigroup operation of $T$
implies that there exists an open neighbourhood $V(\overline{e})$ of
$\overline{e}$ in $T$ such that the following conditions hold:
\begin{equation*}
    \{(a,b)\}\cdot V(\overline{e})\subseteq
    U((a,b)\cdot\overline{e}) \qquad \hbox{ and } \qquad
    V(\overline{e})\cdot\{(a,b)\}\subseteq
    U(\overline{e}\cdot(a,b)).
\end{equation*}
By assertion $(i)$ we get that without loss of generality we can
assume that $V(\overline{e})\cap E(T)$ is a compact subset in $T$
and there exists a positive integer $n_0\geqslant\max\{a,b\}$ such
that $(n,n)\in V(\overline{e})\cap E(T)$ for all integers
$n\geqslant n_0$. Then for $n=2n_0-a$ and $k=2n_0-b$ we get that
$(n,n),(k,k)\in V(\overline{e})\cap E(T)$. But we have
\begin{equation*}
(a,b)\cdot(n,n)=(a,b)\cdot(2n_0-a,2n_0-a)=(2n_0-a-b+a,2n_0-a)=
(2n_0-b,2n_0-a)
\end{equation*}
and
\begin{equation*}
(k,k)\cdot(a,b)=(2n_0-b,2n_0-b)\cdot(a,b)=(2n_0-b,2n_0-b-a+b)=
(2n_0-b,2n_0-a),
\end{equation*}
which contradicts $U((a,b)\cdot\overline{e})\cap
U(\overline{e}\cdot(a,b))= \varnothing$. The obtained contradiction
implies that $(a,b)\cdot\overline{e}=\overline{e}\cdot(a,b)$ for
every $(a,b)\in\mathscr{C}_{\mathbb{Z}}$.

Next we show that $x\cdot\overline{e}=\overline{e}\cdot x$ for every
$x\in T\setminus\mathscr{C}_{\mathbb{Z}}$. Suppose contrary: there
exists $x\in T\setminus\mathscr{C}_{\mathbb{Z}}$ such that
$x\cdot\overline{e}\neq\overline{e}\cdot x$. Then the Hausdorffness
of the space $T$ implies that there exist open neighbourhoods
$U(x\cdot\overline{e})$ and $U(\overline{e}\cdot x)$ of the points
$x\cdot\overline{e}$ and $\overline{e}\cdot x$ in $T$ such that
$U(x\cdot\overline{e})\cap U(\overline{e}\cdot x)= \varnothing$. The
continuity of the semigroup operation of $T$ implies that there
exists an open neighbourhood $V(x)$ of $x$ in $T$ such that the
following conditions hold:
\begin{equation*}
    V(x)\cdot\left\{\overline{e}\right\}\subseteq
    U(x\cdot\overline{e}) \qquad \hbox{ and } \qquad
    \left\{\overline{e}\right\}\cdot V(x)\subseteq
    U(\overline{e}\cdot x).
\end{equation*}
Since $\mathscr{C}_{\mathbb{Z}}$ is a dense subsemigroup of $T$ we
conclude that there exists $(a,b)\in\mathscr{C}_{\mathbb{Z}}$ such
that $(a,b)\in V(x)$. Then we get that $(a,b)\cdot\overline{e}=
\overline{e}\cdot (a,b)$, which contradicts
$U(x\cdot\overline{e})\cap U(\overline{e}\cdot x)= \varnothing$. The
obtained contradiction implies that
$x\cdot\overline{e}=\overline{e}\cdot x$ for every $x\in T$.

We define a map $\mathfrak{h}\colon T\rightarrow I$ by the formula
$(x)\mathfrak{h}=x\cdot\overline{e}$. Since $x\cdot\overline{e}=
\overline{e}\cdot x$ for every $x\in T$ we get that $\mathfrak{h}$
is a homomorphism. Since $\mathscr{C}_{\mathbb{Z}}$ is a dense
subsemigroup of $T$, Proposition~\ref{proposition-2} and assertion
$(iii)$ of Theorem~\ref{theorem-16-0} imply that the topological
group $I$ contains a dense cyclic subgroup. Since $I$ is a locally
compact topological group, Pontryagin-Weil Theorem (see \cite[p.~71,
Theorem~19]{Morris1977}) implies that either $I$ is compact or $I$
is discrete. If $I$ is compact, then by
Proposition~\ref{proposition-10}$(viii)$ we get that
\begin{equation*}
    S=T\setminus\bigcup_{(a,b)\notin
    \mathscr{C}_{\mathbb{N}}[0]}{\uparrow}(a,b)
\end{equation*}
is a closed subset in $T$. Then by Theorem~3.3.9 from
\cite{Engelking1989} $S$ is a locally compact space. Obviously,
$S=\mathscr{C}_{\mathbb{N}}[0]\cup I$. Since $I$ is a locally
compact ideal in $T$, Proposition~\ref{proposition-1}$(viii)$ and
Proposition~II.4 from \cite{EberhartSelden1969} imply that the Rees
quotient semigroup $S/I$ with  the quotient topology is locally
compact topological inverse semigroup which is isomorphic to the
bicyclic semigroup with an adjoined zero. This contradicts
Proposition~V.3 from \cite{EberhartSelden1969}. The obtained
contradiction implies that the group $I$ is discrete and hence $I$
is a discrete additive group of integers.

$(iii)$ Let $(a,b),(c,d)\in\mathscr{C}_{\mathbb{Z}}$ such that
$(a,b)\mathfrak{C}_{mg}(c,d)$. Then there exists an idempotent
$(n,n)\in\mathscr{C}_{\mathbb{Z}}$ such that $(a,b)\cdot(n,n)=
(c,d)\cdot(n,n)$. Since $(i,i)\cdot\overline{e}=\overline{e}$ for
every idempotent $(i,i)\in\mathscr{C}_{\mathbb{Z}}$ we get that
$((a,b))\mathfrak{h}=((c,d))\mathfrak{h}$.

Let $(a,b),(c,d)\in\mathscr{C}_{\mathbb{Z}}$ such that
$((a,b))\mathfrak{h}=((c,d))\mathfrak{h}$. Suppose the contrary:
$(a,b)\cdot(n,n)\neq(c,d)\cdot(n,n)$ for any idempotent
$(n,n)\in\mathscr{C}_{\mathbb{Z}}$. If
$(a,b)\cdot(n_1,n_1)=(c,d)\cdot(n_2,n_2)$ for some idempotents
$(n_1,n_1),(n_2,n_2)\in\mathscr{C}_{\mathbb{Z}}$, then we have that
\begin{equation*}
\begin{split}
  (a,b)\cdot(n_1,n_1)\cdot(n_2,n_2)= &\;
  (a,b)\cdot(n_1,n_1)\cdot(n_1,n_1)\cdot(n_2,n_2)\\
   = &\;
  (c,d)\cdot(n_2,n_2)\cdot(n_1,n_1)\cdot(n_2,n_2)\\
   = &\; (c,d)\cdot(n_1,n_1)\cdot(n_2,n_2).
\end{split}
\end{equation*}
Therefore we get that $(a,b)\cdot(n_1,n_1)\neq(c,d)\cdot(n_2,n_2)$
for all idempotents $(n_1,n_1),(n_2,n_2)\in
\mathscr{C}_{\mathbb{Z}}$. Then
Proposition~\ref{proposition-1}$(vi)$ implies that $b-a\neq d-c$,
and hence by the proof of Proposition~\ref{proposition-2} we get
that the congruence on the semigroup $\mathscr{C}_{\mathbb{Z}}$
which is generated by the homomorphism $\mathfrak{h}$ distincts from
the minimal group congruence $\mathfrak{C}_{mg}$ on
$\mathscr{C}_{\mathbb{Z}}$. Then the ideal $I$ is not isomorphic to
the additive group of integers $\mathbb{Z}$ and hence by
Proposition~\ref{proposition-2} we have that the ideal $I$ contains
a finite cyclic group. This contradicts assertion $(ii)$. The
obtained contradiction implies our assertion.

$(iv)$ Assertions~$(ii)$ and $(iii)$ imply that the subsemigroup
$S=\mathscr{C}_{\mathbb{Z}}\cup I$ of $T$ is algebraically
isomorphic to the inverse semigroup  $S_3$ from
Example~\ref{example-13}. We identify the group $I$ with $G_0$ and
put $\overline{e}=0\in G_0$.

By $\tau$ we denote the topology of the topological inverse
semigroup $T$. Since $G_0$ is a discrete subgroup of $T$,
assertion~$(i)$ implies that there exists a compact open
neighbourhood $U(0)$ of $0$ in $T$ with the following property:
\begin{itemize}
  \item[] $U(0)\subseteq E(T)$ and there is a positive integer
   $n_0$ such that $n_0=\max\{(n,n)\in
   E(\mathscr{C}_{\mathbb{Z}})\mid(n,n)\in U(0)\}$ and $(i,i)\in
   U(0)$ for all integers $i\geqslant n_0$.
\end{itemize}

Hence, we get that $\mathscr{B}_3(0)=\{U_n(0)\mid n\in\mathbb{N}\}$
is a base of the topology of the space $T$ at the point $0\in
G_0\subseteq T$, where $U_n(0)=\{0\}\cup\{(n+i,n+i)\mid
i\in\mathbb{N}\}$.

We fix an arbitrary element $k\in G_0$. Without loss of generality
we can assume that $k\geqslant 0$. Then $k^{-1}=-k\in
\mathbb{Z}=G_0$. Since $G_0$ is a discrete subgroup of $T$, the
continuity of the homomorphism $\mathfrak{h}\colon T\rightarrow
G_0\colon x\mapsto x\cdot\overline{e}=x\cdot 0$ implies that
$(k)\mathfrak{h}^{-1}$ is an open subset in $T$. We observe that,
since the homomorphism $\mathfrak{h}$ generates the minimal group
congruence on $\mathscr{C}_{\mathbb{Z}}$ (see assertion $(iii)$) we
get that $(k)\mathfrak{h}^{-1}\cap \mathscr{C}_{\mathbb{Z}}
=\{(a,b)\in\mathscr{C}_{\mathbb{Z}}\mid b-a=k\}$. Also, since
\begin{equation*}
    {\uparrow}(a,b)=\{(x,y)\in\mathscr{C}_{\mathbb{Z}}\mid
    (x,y)\cdot(b,b)=(a,b)\},
\end{equation*}
for every $(a,b)\in\mathscr{C}_{\mathbb{Z}}$,
Proposition~\ref{proposition-10}$(viii)$ implies that
${\uparrow}(a,b)$ is a closed-and-open subset in $T$ for every
$(a,b)\in\mathscr{C}_{\mathbb{Z}}$. Hence we get that
$\{k\}\cup\{(i,i+k)\in\mathscr{C}_{\mathbb{Z}}\mid i=1,2,3,\ldots\}$
is an open subset in $T$.

We fix an arbitrary positive integer $i$. Since $(i+k,i)\cdot k=0\in
G_0$, the continuity of the semigroup operation in $T$ implies that
for every $U_i(0)\in\mathscr{B}_3(0)$ there exists an open
neighbourhood
\begin{equation*}
V(k)\subseteq \{k\}\cup\{(i,i+k)\in\mathscr{C}_{\mathbb{Z}}\mid
i=1,2,3,\ldots\}
\end{equation*}
of $k$ in $T$ such that $(i+k,i)\cdot V(k)\subseteq U_i(0)$. Then
the semigroup operation of $\mathscr{C}_{\mathbb{Z}}$ implies that
$V(k)\subseteq U_i(k)$ for $U_i(k)\in\mathscr{B}_3(k)$.

We observe that for every $k\in G_0$ and for every positive integer
$i$ we have that
\begin{equation*}
    0\cdot (i,k+i)=k \qquad \hbox{ and } \qquad
    U_i(0)\cdot\{(i,i+k)\}=U_ i(k),
\end{equation*}
where $U_i(0)\in\mathscr{B}_3(0)$ and $U_i(k)\in\mathscr{B}_3(k)$.
Then the continuity of the semigroup operation in $T$ implies that
for every open neighbourhood $W(k)$ of $k$ in $T$ there exists
$U_i(0)\in\mathscr{B}_3(0)$ such that
\begin{equation*}
    U_i(0)\cdot\{(i,i+k)\}=U_ i(k)\subseteq W(k).
\end{equation*}
This implies that the bases of topologies $\tau$ and $\tau_3$ at the
point $k\in T$ coincide.

In the case when $k<0$ the proof is similar. This completes the
proof of our assertion.
\end{proof}

Theorem~\ref{theorem-16} implies the following:

\begin{corollary}\label{corollary-17}
Let $T$ be a Hausdorff locally compact topological inverse
semigroup. If $T$ contains $\mathscr{C}_{\mathbb{Z}}$ as a dense
subsemigroup such that $I=
T\setminus{\uparrow}\mathscr{C}_{\mathbb{Z}}\neq \varnothing$ and
${\uparrow}\mathscr{C}_{\mathbb{Z}}=\mathscr{C}_{\mathbb{Z}}$, then
$T$ is topologically isomorphic to the topological inverse semigroup
$(S_3,\tau_3)$ from Example~\ref{example-13}.
\end{corollary}

\begin{theorem}\label{theorem-18}
Let $(T,\tau)$ be a Hausdorff locally compact topological inverse
monoid with unit $1_T$. If $\mathscr{C}_{\mathbb{Z}}$ is a dense
subsemigroup of $T$ such that ${\uparrow}\mathscr{C}_{\mathbb{Z}}=T$
and the group of units of $T$ is singleton, then there exists a
decreasing sequence of negative integers $\{m_i\}_{i\in\mathbb{N}}$
such that $(T,\tau)$ is topologically isomorphic to the semigroup
$(S_1,\tau_1)$ from Example~\ref{example-11}.
\end{theorem}

\begin{proof}
By the assumption of the theorem we have that $T\setminus
\mathscr{C}_{\mathbb{Z}}=\left\{1_T\right\}$. Then
Lemma~\ref{lemma-8}$(i)$ implies that there exists a base
$\mathscr{B}(1_T)$ of the topology $\tau$ at the unit $1_T$ such
that $U(1_T)\subseteq E(\mathscr{C}_{\mathbb{Z}})$ for any
$U(1_T)\in\mathscr{B}(1_T)$. Also statements $(c)$ and $(d)$ of
Theorem~1.7 from \cite[Vol.~1]{CHK} imply that we can assume that
$(n,n)\in U(1_T)$ if and only if $n$ is a negative integer. Since by
Corollary~\ref{corollary-4} every element of the semigroup
$\mathscr{C}_{\mathbb{Z}}$ is an isolated point of $T$, without loss
of generality we can assume that all elements of the base
$\mathscr{B}(1_T)$ are closed-and-open subsets of $T$. Also, the
local compactness of $T$ implies that without loss of generality we
can assume that the base $\mathscr{B}(1_T)$ consists of compact
subsets, and Corollary~3.3.6 from \cite{Engelking1989} implies that
the base $\mathscr{B}(1_T)$ is countable.

We suppose that $\mathscr{B}(1_T)=\left\{U_n(1_T)\mid
n=1,2,3,\ldots\right\}$. We put
\begin{equation*}
    W_1(1_T)=U_1(1_T) \qquad \hbox{ and } \qquad
    W_i(1_T)=W_{i-1}(1_T)\cap U_i(1_T),
\end{equation*}
for all $i=2,3,4,\ldots$. We observe that
$\widetilde{\mathscr{B}}(1_T)=\left\{W_n(1_T)\mid
n=1,2,3,\ldots\right\}$ is a base of the topology $\tau$ at the unit
$1_T$ of $T$ such that $W_{n+1}(1_T)\subsetneqq W_n(1_T)$ for every
positive integer $n$. Then the compactness of $U_i(1_T)$,
$i=1,2,3,\ldots$, and the discreteness of the space
$\mathscr{C}_{\mathbb{Z}}$ imply that the family
$\widetilde{\mathscr{B}}(1_T)$ consists of compact-and-open subsets
of $T$. Let $\left\{m_i\right\}_{i\in\mathbb{N}}$ be a decreasing
sequence of negative integers such that
$\displaystyle\bigcup_{i=1}^\infty\left\{(m_i,m_i)\right\}=
W_1(1_T)\setminus\left\{ 1_T\right\}$. We put
$V_n=\left\{1_T\right\}\cup
\left\{(m_i,m_i)\in\mathscr{C}_{\mathbb{Z}}\mid i\geqslant
n\right\}$ for every positive integer $n$. Since every element of
the family $\widetilde{\mathscr{B}}(1_T)$ is a compact subset of
$T$, Corollary~\ref{corollary-4} implies that the family
\begin{equation*}
    \overline{\mathscr{B}}(1_T)=
    \left\{V_n\mid n=1,2,3,\ldots\right\}
\end{equation*}
is a base of the topology $\tau$ at $1_T$ of $T$ and this completes
the proof of our theorem.
\end{proof}

Theorems~\ref{theorem-16} and ~\ref{theorem-18} imply the following:

\begin{corollary}\label{corollary-19}
Let $(T,\tau)$ be a Hausdorff locally compact topological inverse
semigroup. If $\mathscr{C}_{\mathbb{Z}}$ is a dense subsemigroup of
$T$ such that the group of units of $T$ is singleton, then there
exists a decreasing sequence of negative integers
$\{m_i\}_{i\in\mathbb{N}}$ such that $(T,\tau)$ is topologically
isomorphic either to the semigroup $(S_1,\tau_1)$ from
Example~\ref{example-11} or to the semigroup $(S_4,\tau_4)$ from
Example~\ref{example-14}.
\end{corollary}

\begin{theorem}\label{theorem-20}
Let $(T,\tau)$ be a Hausdorff locally compact topological inverse
monoid with unit $1_T$. Suppose that $\mathscr{C}_{\mathbb{Z}}$ is a
dense subsemigroup of $T$ such that the following conditions hold:
\begin{itemize}
  \item[$(i)$] ${\uparrow}\mathscr{C}_{\mathbb{Z}}=T$;
  \item[$(ii)$] the group of units $H(1_T)$ of $T$ is non-singleton;
   \; and
  \item[$(iii)$] there exists an integer $j$ such that
   $K=\left\{1_T\right\}\cup\left\{(i,i)\in\mathscr{C}_{\mathbb{Z}}
   \mid i\geqslant j\right\}$ is a compact subset of $T$.
\end{itemize}
Then there exists a decreasing sequence of negative integers
$\{m_i\}_{i\in\mathbb{N}}$ such that $m_{i+1}=m_i-1$ for every
positive integer $i$ and $(T,\tau)$ is topologically isomorphic to
the semigroup $(S_2,\tau_2)$ for $n=1$ from
Example~\ref{example-12}.
\end{theorem}

\begin{proof}
As in the proof of Theorem~\ref{theorem-18} we construct a
decreasing sequence of negative integers $\{m_i\}_{i\in\mathbb{N}}$
such that the family
\begin{equation*}
    \mathscr{B}(1_T)=
    \left\{U_i(1_T)\mid i=1,2,3,\ldots\right\}
\end{equation*}
determines a base of the topology $\tau$ at the point $1_T$ of $T$,
where
\begin{equation*}
    U_j(1_T)=\left\{1_T\right\}\cup
   \left\{(m_i,m_i)\in\mathscr{C}_{\mathbb{Z}}\mid i\geqslant
   j\right\}.
\end{equation*}
The compactness of the set $K$ implies that we can construct a
sequence of negative integers $\{m_i\}_{i\in\mathbb{N}}$ such that
$m_{i+1}=m_i-1$ for every positive integer $i$.

Then for every element $x$ of the group of units $H(1_T)$ left and
right translations $\lambda_x\colon T\rightarrow T: s\mapsto x\cdot
s$ and $\rho_x\colon T\rightarrow T: s\mapsto s\cdot x$ are
homeomorphisms of the topological space $T$ (see \cite[Vol.~1,
P.~19]{CHK}), and hence the following families
\begin{equation*}
    \mathscr{B}_l(x)=\left\{x\cdot U_i(1_T)\mid
    U_i(1_T)\in\mathscr{B}(1_T)\right\} \qquad
\end{equation*}
     {and}
\begin{equation*}
     \qquad
    \mathscr{B}_r(x)=\left\{U_i(1_T)\cdot x\mid
    U_i(1_T)\in\mathscr{B}(1_T)\right\}
\end{equation*}
are bases of the topology $\tau$ at the point $1_T$ of $T$. Also, we
observe that the family
\begin{equation*}
    \mathscr{B}(x)=\left\{U\cap V\mid
    U\in\mathscr{B}_l(x) \hbox{ and }
    V\in\mathscr{B}_r(x)\right\}
\end{equation*}
is a base of the topology $\tau$ at the point $1_T$ of $T$.

Then Lemma~\ref{lemma-8} and Proposition~\ref{proposition-10} imply
that the group of units $H(1_T)$ of $T$ is topologically isomorphic
to the discrete additive group of integers $\mathbb{Z}_+$. Let $g$
be a generator of $\mathbb{Z}_+$. Then by Lemma~\ref{lemma-8}$(iii)$
there exist an open neighbourhood $U(g)$ of the point $g$ in $T$ and
an integer $k$ such that $a-b=k$ for all $(a,b)\in U(g)\cap
\mathscr{C}_{\mathbb{Z}}$. Without loss of generality we can assume
that $g$ is a positive integer and $k<0$. Then we have that
\begin{equation}\label{f-3}
    g\cdot U_i(1_T)=\left\{(m_i+k,m_i)\mid (m_i,m_i)\in
    U_i(1_T)\right\}\cup\left\{g\right\}
\end{equation}
and
\begin{equation}\label{f-4}
    U_i(1_T)\cdot g=\left\{(m_i,m_i-k)\mid (m_i,m_i)\in
    U_i(1_T)\right\}\cup\left\{g\right\}
\end{equation}
We shall show that equality (\ref{f-4}) holds. Let be $(m_i,m_i)\in
U_i(1_T)$. Then we get
\begin{equation*}
     \left((m_i,m_i)\cdot g\right)\cdot\left((m_i,m_i)\cdot
     g\right)^{-1}=(m_i,m_i)\cdot g\cdot g^{-1}\cdot(m_i,m_i)^{-1}=
     (m_i,m_i)\cdot 1_T\cdot(m_i,m_i)=(m_i,m_i).
\end{equation*}
Since $(m_i,m_i)\cdot g\in \mathscr{C}_{\mathbb{Z}}$ and
$\mathscr{C}_{\mathbb{Z}}$ is an inverse semigroup we conclude that
$(m_i,m_i)\cdot g=(m_i,a)$ for some integer $a$, and by
Lemma~\ref{lemma-8}$(vi)$ we have that $(m_i,m_i)\cdot
g=(m_i,m_i-k)$. This completes the proof of equality (\ref{f-4}).
The proof of equality (\ref{f-3}) is similar. Then
Lemma~\ref{lemma-8}$(vi)$, equalities (\ref{f-3}) and (\ref{f-4})
imply that $T$ is topologically isomorphic to the semigroup
$(S_2,\tau_2)$ for $n=1$ from Example~\ref{example-12}. This
completes the proof of the theorem.
\end{proof}

Theorems~\ref{theorem-16} and ~\ref{theorem-20} imply the following:

\begin{corollary}\label{corollary-21}
Let $(T,\tau)$ be a Hausdorff locally compact topological inverse
monoid with unit $1_T$. Suppose that $\mathscr{C}_{\mathbb{Z}}$ is a
dense subsemigroup of $T$ such that the following conditions hold:
\begin{itemize}
  \item[$(i)$] the group of units $H(1_T)$ of $T$ is non-singleton;
   \; and
  \item[$(ii)$] there exists an integer $j$ such that
   $K=\left\{1_T\right\}\cup\left\{(i,i)\in\mathscr{C}_{\mathbb{Z}}
   \mid i\geqslant j\right\}$ is a compact subset of $T$.
\end{itemize}
Then there exists a decreasing sequence of negative integers
$\{m_i\}_{i\in\mathbb{N}}$ such that $m_{i+1}=m_i-1$ for every
positive integer $i$ and $(T,\tau)$ is topologically isomorphic
either to the semigroup $(S_2,\tau_2)$ from Example~\ref{example-12}
or to the semigroup $(S_5,\tau_5)$ from Example~\ref{example-15}.
\end{corollary}

\end{document}